\documentclass{article}
\usepackage[T1]{fontenc}
\usepackage{titlesec}
\usepackage[utf8]{inputenc}
\usepackage{amsmath}
\usepackage{amssymb}
\usepackage{amsfonts}
\usepackage{amsthm}
\usepackage{geometry}
\usepackage{babel}
\usepackage{xcolor}
\newtheorem{lemma}{Lemma}
\newtheorem{theorem}{Theorem}
\newtheorem{proposition}{Proposition}

\newtheorem*{remark}{Remark}

\numberwithin{equation}{section}
\usepackage{url}
\usepackage[hidelinks]{hyperref}

\title{PhD}
\author{}
\date{January 2021}

\begin{document}

\begin{center}
    \textbf{On Empirical Spectral Distributions for Random Tensor Product Models}
\end{center}

\begin{center}
     Simona Diaconu\footnote{Courant Institute, New York University, simona.diaconu@nyu.edu}
\end{center}

\begin{abstract}
    In statistics, assuming samples are independent is reasonable. However, this property can fail to hold for the features, 
    a distinction that has led to several lines of work aiming to remove the latter assumption of independence 
    present in the early literature, while preserving the original conclusions. 
    Empirical spectral distributions of 
    covariance matrices are key for understanding the data, and their almost sure convergence is oftentimes
    desirable. The random tensor product model, \(X=(x_{i_1}x_{i_2}...x_{i_d})_{1 \leq i_1<...<i_d \leq n}\) 
    for \(x_1,x_2,\hspace{0.05cm}...\hspace{0.05cm},x_n\) i.i.d., introduced by the machine learning community, has a dependence structure for its features far from trivial and has been studied 
    in recent years. When \(x_1 \in \mathbb{R}, \mathbb{E}[x_1^4]<~\infty, \frac{d}{n^{1/3}}=~o(1),\) the empirical spectral distributions of the covariance matrices were proved to converge almost surely to Marchenko-Pastur laws in the random matrix theory regime. This work extends this result to the range \(\frac{d}{n^{1/2}}=~o(1)\) when \(x_1\) is symmetric with a subgaussian norm slowly growing in \(n\) 
    (the aforesaid range arises naturally, and the result failing when \(\frac{d}{n^{1/2}} \to \infty\) appears to be a plausible claim) and shows that similarly to the case with independent features, 
    the almost sure convergence holds under more general conditions on the covariance structure than the isotropic case. The latter result 
    provides a means of deriving convergence for empirical spectral distributions of random matrices, applicable 
    to other models as well so long as their entries exhibit a certain degree of concentration. 
\end{abstract}


\section{Introduction}

For a random matrix\footnote{This paper deals solely with symmetric real-valued random matrices. However, the results could be easily extended to the complex Hermitian random matrices as well.} \(A \in \mathbb{R}^{n \times n},A=A^T,\) its empirical spectral distribution is the function \(F_{A}:\mathbb{R} \to [0,1],\) 
\[F_A(x)=\frac{1}{n}\sum_{1 \leq j \leq n}{\chi_{x \geq \lambda_j(A)}},\]
where \(\lambda_1(A) \geq \lambda_2(A) \geq ... \geq \lambda_n(A)\) are the eigenvalues of \(A.\) Sample covariance matrices, among the most pervasive tools in statistics, are given by 
\begin{equation}\tag{\(M\)}\label{covar}
    A=\frac{1}{m}XX^T \in \mathbb{R}^{n \times n},
\end{equation}
where the columns of \(X \in \mathbb{R}^{n \times m}\) are independent copies of a vector \(X_0 \in \mathbb{R}^n,\) whose entries are oftentimes called features. The function \(F_A\) remains an object of interest due to its ability not only of capturing the global behavior of spectrum, but also of encapsulating local features such as the asymptotics of the leading eigenvalues (e.g., Bai and Yin~\cite{baiyin}, Bai and Silverstein~\cite{baisilv1s},~\cite{baisilv2s}, Bai et al.~\cite{baietal}). Marchenko and Pastur in their seminal work \cite{marchenkopastur} give sufficient conditions on the structure of Hermitian matrices that guarantee the weak convergence of their empirical spectral distributions to deterministic measures in probability. Several subsequent papers have relaxed the original assumptions on the random variables underlying the entries of \(A:\) namely, for covariance matrices given by (\ref{covar}), the conditions in \cite{marchenkopastur} that the entries of \(X_0\) are i.i.d., centered and have a finite fourth moment have been weakened considerably. It is now well established that when the entries of \(X_0\) are independent, all variances being finite is a sufficient and necessary condition for convergence of \(F_A\) (Belinschi et al.~\cite{demboetal} treat a family of models with infinite second moments and show almost sure weak convergence of the empirical spectral distributions, but under a different scaling than that in (\ref{covar}): their theorem \(1.1\) showcases striking differences between these limits and the Marchenko-Pastur laws (\ref{mplaws})), 
while the 
identically distributed assumption can be replaced by the features being a linear transformation of i.i.d. random variables contingent on the underlying matrix satisfying a certain asymptotic behavior expressed also in terms of empirical spectral distributions (see Silverstein~\cite{silv1}, Silverstein and Bai~\cite{silv2}, as well as Theorem~\ref{mainth2} below). Furthermore, efforts to relax the independence assumption have also been successful, some replacements being martingale conditions, \(m\)-independence, block models to name but a few (see subsection \(1.2\) in Bryson et al.~\cite{bryson} for a more detailed account, including references). 
\par
This work focuses on the random tensor product model, popularized by machine learning architectures and analyzed in the random matrix theory literature in recent years. Concretely, let \(Z=Z(n) \in \mathbb{R}^{N \times p}\) be a random matrix whose \(p\) columns are i.i.d., 
\begin{equation}\tag{RMT}\label{modelz}
    d=d(n),\hspace{1cm} N=N(n)=\binom{n}{d}, \hspace{1cm} p=p(n), \hspace{1cm} \gamma_n=\frac{N}{p} \to \gamma>0,
\end{equation}
and under a bijection \(\{(i_1,i_2,\hspace{0.05cm}...\hspace{0.05cm},i_d): 1 \leq i_1<i_2<...<i_d \leq n\} \to \{1,2,\hspace{0.05cm}...\hspace{0.05cm},N\},\) each column of \(Z\) is equal in distribution to 
\begin{equation}\tag{RTP}\label{zmoddel}
    Z_0 \in \mathbb{R}^{\binom{n}{d}}, \hspace{0.5cm} Z_{0,i_1,i_2,\hspace{0.05cm}...\hspace{0.05cm},i_d}=x_{i_1}x_{i_2}...x_{i_d} \hspace{0.5cm} (1 \leq i_1<i_2<...<i_d \leq n)
\end{equation} 
for \((x_{j})_{1 \leq j \leq n}\) i.i.d., \(\mathbb{E}[x_{1}]=0,\mathbb{E}[x_{1}^2]=1.\) Bryson et al.~\cite{bryson} show that the empirical spectral distribution of 
\[S=S(n)=\frac{1}{p}ZZ^T \in \mathbb{R}^{N \times N}\] 
converges almost surely to the Marchenko-Pastur distribution with parameter \(\gamma \in (0,\infty),\) i.e., 
\begin{equation}\tag{\(MP\)}\label{mplaws}
    F_\gamma(x)=(1-\frac{1}{\gamma})\chi_{x \geq 0,\gamma>1}+\int_{-\infty}^{x}{p_{\gamma}(y)dy}, \hspace{0.3cm} p_{\gamma}(y)=\begin{cases}
    \frac{\sqrt{((1+\sqrt{\gamma})^2-y) \cdot (y-(1-\sqrt{\gamma})^2)}}{2\pi \gamma y}, & y \in [(1-\sqrt{\gamma})^2,(1+\sqrt{\gamma})^2]\\
0, & y \not \in [(1-\sqrt{\gamma})^2,(1+\sqrt{\gamma})^2]
\end{cases},
\end{equation}
(see subchapter \(3.1\) in Bai and Silverstein~\cite{baisilvbook} for details) so long as\footnote{For ease of notation, the dependence of \(d,p,N\) on \(n\) is dropped in the remainder of this paper.} 
\[\frac{d}{n^{1/3}}=o(1), \hspace{0.5cm} \mathbb{E}[x_{11}^4]<\infty,\] 
and the authors conjecture the result remains valid when \(\frac{d}{n^{1/2}}=o(1).\) This latter condition appears naturally after looking at the lengths of the columns of \(Z:\) it can be shown that 
\begin{equation}\tag{\(V1\)}\label{varz0}
    Var(\frac{1}{N}||Z_1||^2) \propto \frac{d^2}{n}
\end{equation}
when \(\mathbb{E}[x_1^4]<\infty,\frac{d}{n^{1/2}}=o(1)\) (see Lemma~\ref{lenlemma} and (\ref{lbvar}) herein for full details: these results are tight in the sense that when \(d=\Omega(n^{1/2}),\) the variance can increase exponentially in \(\frac{d^2}{n},\) a statement quantified by Lemma~\ref{lenlemma2}), and variances that vanish asymptotically have oftentimes been used as sufficient conditions for weak convergences of empirical spectral distributions of random matrices. In particular, early extensions of the pioneering work of Marchenko  and Pastur~\cite{marchenkopastur} such as Silverstein~\cite{silv1}, Bai and Silverstein~\cite{silv2}, as well as more recent developments (e.g., Bai and Zhou~\cite{baizhou}) require 
\begin{equation}\tag{\(V2\)}\label{varz00}
    Var(\frac{1}{N}Z_1^TAZ_1)=o(1), \hspace{0.5cm} \forall A \in \mathbb{R}^{N \times N}, A=A^T, ||A|| \leq 1.
\end{equation}
It is known nevertheless that condition (\ref{varz00}) is not necessary (e.g., Götze and Tikhomirov~\cite{gotze1},~\cite{gotze2}): its presence allows for a somewhat direct use of the Stieltjes transform, and avoiding it oftentimes requires an entirely brand new proof technique. 
For instance, Götze and Tikhomirov in both \cite{gotze1} and \cite{gotze2} employ a differential characterization of the Marchenko-Pastur laws alongside the Stieltjes transform, while Adamczak~\cite{adamczac} uses a determinant condition (based on Tao and Vu~\cite{taovu}) and concentration of measure (see theorem \(3.9\) in \cite{adamczac}). It is worthwhile mentioning that since (\ref{varz00}) is stronger than (\ref{varz0}), it is not clear whether the latter is also unnecessary: if it were necessary, then it would indicate that the range \(\frac{d}{n^{1/2}}=o(1)\) might be optimal for the random tensor product model given by (\ref{zmoddel}) (Lemma~\ref{lenlemma2} treats solely part of the range \(\frac{d}{n^{1/2}}=\Omega(1)\)). 
\par
This paper confirms the aforesaid conjecture from Bryson et al.~\cite{bryson} under two additional assumptions on the law of \(x_1:\) symmetry and a slowly growing subgaussian norm. The main ingredient remains the Stieltjes transform, which together with concentration of measure and uniqueness of analytical continuations renders two extensions of the convergence for the random tensor product model justified in~\cite{bryson}. The first result is based on 
identity (\ref{554}) for the Stieltjes transform of \(S,\) previously employed, for instance, in Silverstein~\cite{silv1}, Ledoit and Péché~\cite{ledoitpeche}, as well as in one of the author's former works,~\cite{oldpaper}. 


\begin{theorem}\label{mainth}
Suppose \(Z \in \mathbb{R}^{N \times p}\) has i.i.d. columns with distributions given by (\ref{zmoddel}) with \(d<n^{1/2},\)
\[x_1=x(n), \hspace{0.5cm} x_1 \overset{d}{=} -x_1,\hspace{0.5cm} \mathbb{E}[x_1^2]=1, \hspace{0.5cm} \mathbb{E}[x_1^{2k}] \leq (Ck)^k \hspace{0.5cm} (k \in \mathbb{N}), \hspace{0.5cm} \lim_{n \to \infty}{\frac{\log{C}}{\log{\frac{n^{1/2}}{d}}}}=0\]
for some \(C=C(n) \geq 2.\)
If (\ref{modelz}) holds, 
then the empirical spectral distribution of \(S=\frac{1}{p}ZZ^T\) converges weakly to the Marchenko-Pastur law \(F_{\gamma},\) given by (\ref{mplaws}), 
almost surely.
\end{theorem}

\par
A few comments on the statement above are in order: the lower bound on \(C\) is arbitrary in the sense that \(2\) can be replaced by \(1+\epsilon_0\) for any \(\epsilon_0>0\) and so is the upper bound on \(\frac{d}{n^{1/2}}.\) These two inequalities entail that the limit condition on \(C\) is equivalent to\footnote{Contingent on \(d<n^{1/2},C \geq 2:\) the equivalence in Lemma~\ref{lemmaequiv} below does not hold when \(b_n \to 0,\) e.g., \(a_n=2^n,b_n=2^{-n^2}.\)} \(\frac{d}{n^{1/2}}=o(1).\) 
Furthermore, the following simple lemma illustrates the utility of the latter assumption: 
the existence of a sequence \(q_n \in \mathbb{N}\) as described below for \(a_n=C(n),b_n=\frac{n^{1/2}}{d}\) will be vital in the proof of Theorem~\ref{mainth}.

\begin{lemma}\label{lemmaequiv}
    Suppose \(a_n \geq 2, b_n >1.\) Then \(\lim_{n \to \infty}{\frac{\log{a_n}}{\log{b_n}}}=0\) if and only if there exists a sequence \(q_n \in \mathbb{N}\) with \(\lim_{n \to \infty}{q_n}=\infty, \lim_{n \to \infty}{\frac{a_n^{q_n}}{b_n}}=0.\)
\end{lemma}

\begin{proof}
    Suppose \(\lim_{n \to \infty}{\frac{\log{a_n}}{\log{b_n}}}=0.\) This and \(a_n \geq 2,b_n > 1\) imply \(\lim_{n \to \infty}{\log{b_n}}=\infty,\) whereby \(a_n=b_n^{\epsilon_n}\) with \(\epsilon_n>0, \lim_{n \to \infty}{\epsilon_n}=0.\) Let \(q_n=1+\lfloor \frac{1}{2\epsilon_n} \rfloor.\) Then \(q_n \leq \frac{1}{2\epsilon_n}+1 \leq \frac{2}{3\epsilon_n}\) for \(n \geq n_0,\) whereby 
    \[0<\frac{a_n^{q_n}}{b_n}= b_n^{\epsilon_n \cdot q_n-1} \leq b_n^{-1/3}\]
    gives \(\lim_{n \to \infty}{\frac{a_n^{q_n}}{b_n}}=0,\) and \(q_n \in \mathbb{N}, \lim_{n \to \infty}{q_n}=\infty\) are immediate from the definition of this sequence.
    \par
    Conversely, suppose \(q_n \in \mathbb{N},\) \(\lim_{n \to \infty}{q_n}=\infty,\) and \(\lim_{n \to \infty}{\frac{a_n^{q_n}}{b_n}}=0.\) 
    Then for \(n \geq n_0,\)
    \[0<\frac{\log{a_n}}{\log{b_n}} \leq \frac{1}{q_n},\] 
    because \(\lim_{n \to \infty}{(q_n\log{a_n}-\log{b_n})}=-\infty,\) 
    whereby \(\lim_{n \to \infty}{\frac{\log{a_n}}{\log{b_n}}}=0.\)
\end{proof}

The second result is a generalization of Theorem~\ref{mainth} (\(T=I\) corresponds to Theorem~\ref{mainth}): it relies on a key idea employed in the pioneering work~\cite{marchenkopastur} by Marchenko and Pastur, the integral equations satisfied by the Stieltjes transforms of the limiting measures have unique solutions\footnote{Subject to \(m(z) \in \{r \in \mathbb{C}: -\frac{1-c}{z}-cr \in \mathbb{C}^{+}\}\) for \(z \in \mathbb{C}^+.\)}.
Additionally, universality of expectations is crucial for this result, and it must be mentioned that the proof method can be employed to other models contingent on a certain concentration of measure.

\begin{theorem}\label{mainth2}
Suppose \(n,d \in \mathbb{N}, d<n^{1/2},\)
\[x_1=x(n), \hspace{0.5cm} x_1 \overset{d}{=} -x_1,\hspace{0.5cm} \mathbb{E}[x_1^2]=1, \hspace{0.5cm} \mathbb{E}[x_1^{2k}] \leq (Ck)^k \hspace{0.5cm} (k \in \mathbb{N}), \hspace{0.5cm} \lim_{n \to \infty}{\frac{\log{C}}{\log{\frac{n^{1/2}}{d}}}}=0\]
for some \(C=C(n) \geq 2.\) Let \(T=T(n) \in~\mathbb{R}^{N \times N}\) be deterministic, symmetric, positive semidefinite with \(\sup_{n \in \mathbb{N}}{||T(n)||}<\infty\) and its empirical spectral distribution converging weakly to a deterministic probability distribution \(H.\) 
If (\ref{modelz}) holds, 
then the empirical spectral distribution of \(\tilde{S}=\frac{1}{p}\tilde{Z}\tilde{Z}^T,\) where \(\tilde{Z} \in \mathbb{R}^{N \times p}\) has i.i.d. columns equal in law to \(T^{1/2}Z_0\) for \(Z_0 \in \mathbb{R}^{N}\) given by (\ref{zmoddel}), converges weakly to a probability measure \(\mu\) almost surely, and the Stieltjes transform of \(\mu\) satisfies 
\begin{equation}\tag{\(IE\)}\label{steq}
    m(z)=\int_{\mathbb{R}}{\frac{1}{z-t}d\mu(t)}=\int_{\mathbb{R}}{\frac{1}{z-t(1-\gamma+\gamma z m(z))}dH(t)} \hspace{1cm} (z \in \mathbb{C}^+).
\end{equation}
\end{theorem}

The remainder of the paper is organized as follows. Section~\ref{sectplan} presents the main components of the proof of Theorem~\ref{mainth} (they also form the backbone of Theorem~\ref{mainth2}), section~\ref{sectdetails} consists of the justification of the main ingredient in the aforesaid argument, section~\ref{sectgen} contains the proof of Theorem~\ref{mainth2}, and section~\ref{sectauxlemmas} gathers a few lemmas needed along the way.


\section{Revised 3-Step Method}\label{sectplan}

For a symmetric matrix \(A \in \mathbb{R}^{n \times n},\) its Stieltjes transform is \(m:\mathbb{C}^{+} \to \mathbb{C}\) given by\footnote{Several authors define the Stieltjes transform as \(m(z)=\frac{1}{n}tr((A-zI)^{-1}):\) the two formulations are clearly equivalent up to a sign change.}
\[m(z)=\frac{1}{n}tr((zI-A)^{-1}),\]
where \(\mathbb{C}^{+}=\{z \in \mathbb{C}: Im(z)>0\},\) and it is well-known that weak convergence of empirical spectral distributions ensues from the pointwise convergence of the Stieltjes transforms. 
In particular, the probabilities underlying the empirical spectral distributions of \(A_n\) converge weakly to a measure \(\mu\) almost surely when
\[m_{A_n}(z)=\frac{1}{n}tr((zI-A_n)^{-1}) \to m_{\mu}(z)=\int{\frac{1}{z-x}d\mu(x)} \hspace{0.5cm} (z \in \mathbb{C}^{+})\]
with probability one. This has yielded a 3-step method that has been repeatedly used throughout the literature to justify the aforesaid claim on empirical spectral distributions of random matrices:
\vspace{0.3cm}
\par
\(1.\) \(m_{A_n}(z)-\mathbb{E}[m_{A_n}(z)] \xrightarrow[]{a.s.} 0\) for all \(z \in \mathbb{C}^{+},\)
\vspace{0.3cm}
\par
\(2.\) \(\mathbb{E}[m_{A_n}(z)] \to m_{\mu}(z)\) for all \(z \in \mathbb{C}^{+},\)
\vspace{0.3cm}
\par
\(3.\) \(m_{A_n}(z)-m_\mu(z) \to 0\) for all \(z \in \mathbb{C}^{+}\) almost surely.\vspace{0.3cm}\(\newline\)
These are the crucial ideas in the proof of 
\begin{center}
    \(\lim_{n \to \infty}{m_{A_n}(z)}=m_{\mu}(z)\) for all \(z \in \mathbb{C}^{+}\) almost surely
\end{center} 
as presented in Bai and Silverstein~\cite{baisilvbook}. This \(3\)-step technique has been employed in extensions of the seminal results in Marchenko and Pastur~\cite{marchenkopastur} (e.g., dispensing with the finite fourth moment assumption, improving convergence in probability to almost sure convergence), including a recent relaxation of the original condition of the columns of \(X\) having independent entries by Bai and Zhou~\cite{baizhou}. 
\par
In the current case (i.e., the setup of Theorem~\ref{mainth}), \(1.\) remains unchanged, \(2.\) is justified on a smaller domain, 
\begin{equation}\label{dgamma}
    D(\gamma)=\{z \in \mathbb{C}: Im(z) \geq \frac{16}{9}+4\gamma\} \subset \mathbb{C}^{+},
\end{equation}
and \(3.\) consequently needs some slight adjustments as well. Up to a great extent, the arguments for Theorems~\ref{mainth} and \ref{mainth2} are founded on a simple well-known property of analytical continuations, 
uniqueness. This allows the relaxation aforementioned in \(2.,\) while the domain shrinkage \(\mathbb{C}^{+} \to D(\gamma)\) permits tackling the expectations of interest from a different angle, one based on power series expansions. Before justifying the three stages above, a rearrangement of the Stieltjes transform, used in Silverstein~\cite{silv1}, Ledoit and Péché~\cite{ledoitpeche}, Diaconu~\cite{oldpaper},
is looked at next.

\subsection*{Stieltjes Transform}

Let \(m_n:\mathbb{C}^{+} \to \mathbb{C}\) 
be the Stieltjes transform of the Marchenko-Pastur law with parameter \(\gamma_n,\) i.e.,
\[m_n(z)=\frac{z+\gamma_n-1-\sqrt{(z-\gamma_n+1)^2-4z}}{2\gamma_n z} \hspace{0.5cm} (z \in \mathbb{C}^+),\]
where the square root is given by the branch of the complex logarithm that ensures \(Im(m_n(z))<0\) (alternatively, \(m_n(z)\) is the unique solution to (\ref{eqmp}) with \(Im(m_n(z))<0:\) see subchapter \(3.3\) in Bai and Silverstein~\cite{baisilvbook} for details), while \(m_{F_n}\) is the Stieltjes transform of the empirical spectral distribution of \(S \in \mathbb{R}^{N \times N},\)
\[m_{F_n}(z)=\frac{1}{N}tr((zI-S)^{-1}) \hspace{0.5cm} (z \in \mathbb{C}^+).\]
\par
Denote the columns of \(Z\) by \(Z_1,Z_2, \hspace{0.05cm} ... \hspace{0.05cm}, Z_{p} \in \mathbb{R}^{N}.\) In \cite{oldpaper}, it is shown that when \(z \in \mathbb{R}, z \to \infty,\)
\begin{equation}\label{53}
    \sqrt{n} \cdot z(m_n(z)-m_{F_n}(z)) \xrightarrow[]{p} 0,
\end{equation}
under the assumption that the entries in \(Z\) are i.i.d. with bounded subgaussian norm, i.e.,
\[||Z_{11}||_{\psi_2}=\inf{\{t>0, \mathbb{E}[e^{Z_{11}^2/t^2}] \leq 2\}} \leq C.\]
Identity
\begin{equation}\tag{\(ST\)}\label{554}
    -1+zm_{F_n}(z)=-\frac{1}{\gamma_n}
    +\frac{1}{N}\sum_{1 \leq j \leq p}{\frac{1}{1-\frac{1}{p}Z_j^TR_jZ_j}} ,\hspace{0.5cm}
\end{equation}
where \(R_j=(zI-S+\frac{1}{p}Z_jZ_j^T)^{-1}\) for \(1 \leq j \leq p,\) is fundamental in the proof of (\ref{53}) (see Lemma~\ref{ledoitpechelemma} for proof). 
\begin{equation}\label{zeq}
    z(m_{F_n}(z)-m_n(z))=\frac{1}{\gamma_n} \cdot (\frac{1}{1-\gamma_n m_{F_n}(z)}-\frac{1}{1-\gamma_n m_n(z)})+\Delta_1+\Delta_2,
\end{equation}
where
\[\Delta_1=\frac{1}{N}\sum_{1 \leq j \leq p}{(\frac{1}{1-\frac{1}{p}Z_j^TR_jZ_j}-\frac{1}{1-\frac{1}{p}tr(R_j)})}, \hspace{0.3cm} \Delta_2=\frac{1}{N}\sum_{1 \leq j \leq p}{(\frac{1}{1-\frac{1}{p}tr(R_j)}-\frac{1}{1-\frac{1}{p}tr((zI-S)^{-1})})}\]
insofar as \(m_n(z)=\frac{z+\gamma_n-1-\sqrt{(z-\gamma_n+1)^2-4z}}{2\gamma_n z}\) leads to
\begin{equation}\label{55}
    -1+zm_n(z)=-\frac{1}{\gamma_n}+\frac{1}{\gamma_n} \cdot \frac{1}{1-\gamma_n m_n(z)}
\end{equation}
from
\[(z-\gamma_n+1)^2-4z=(z+\gamma_n-1)^2-4\gamma_nz,\]
providing
\begin{equation}\label{eqmp}
    \gamma_n zm^2_n(z)-m_n(z)(z+\gamma_n-1)+1=0,
\end{equation}
or equivalently 
\[(\gamma_n m_n(z)-1) \cdot (zm_n(z)-1)=-m_n(z),\]
since \(B_1^2-2AB_1+B_2=0\) when \(A-B_1=\sqrt{A^2-B_2},\) an identity yielding (\ref{eqmp}) by plugging in
\[B_1=2\gamma_n zm_n(z), \hspace{0.5cm} A=z+\gamma_n-1 , \hspace{0.5cm} B_2=4\gamma_nz.\] 
\par
Formula (\ref{554}), which holds when \(z \in \mathbb{C}^{+},\) is used to justify the three modified steps mentioned at the beginning of this section. These three phases are detailed next.

\subsection*{First Step}

A close look at the proof of 
\begin{equation}\tag{s1}\label{as}
    m_{F_n}(z)-\mathbb{E}[m_{F_n}(z)] \xrightarrow[]{a.s.} 0, \hspace{0.5cm} \forall z \in \mathbb{C}^{+}
\end{equation}
in Bai and Silverstein~\cite{baisilvbook} reveals that the columns \((Z_j)_{1 \leq j \leq p}\) being independent suffices for this convergence. The argument is presented below for the sake of completeness.
\par
Fix \(z \in \mathbb{C}^+:\) let \(\mathcal{F}_{k}=\sigma(Z_j,k<j\leq p)\) for \(0 \leq k \leq p-1,\) and \(\mathcal{F}_p=\{\emptyset,\Omega\}.\) Then
\begin{equation}\label{mgdiff}
    m_{F_n}(z)-\mathbb{E}[m_{F_n}(z)]=\sum_{1 \leq k \leq p}{(\mathbb{E}[m_{F_n}(z)|\mathcal{F}_{k-1}]-\mathbb{E}[m_{F_n}(z)|\mathcal{F}_k])}:=\sum_{1 \leq k \leq p}{N^{-1}\delta_k},
\end{equation}
and observe that by virtue of \(\mathcal{F}_{k} \subset \mathcal{F}_{k-1}\) for \(1 \leq k \leq p,\) the random variables \((\delta_{p+1-k})_{1 \leq k \leq p}\) are martingale differences with respect to the filtration \((\mathcal{F}_{p+1-k})_{1 \leq k \leq p}.\) Because \(r_k(z)=\frac{1}{N}tr(R_k)\) and \(Z_k\) are independent, it ensues that
\[N^{-1}\delta_k=\mathbb{E}[m_{F_n}(z)-r_k(z)|\mathcal{F}_{k-1}]-\mathbb{E}[m_{F_n}(z)-r_k(z)|\mathcal{F}_k],\]
whereby \(|\delta_k| \leq \frac{2}{Im(z)}\) as
\begin{equation}\label{res1}
    N \cdot [m_{F_n}(z)-r_k(z)]=tr((zI-S)^{-1})-tr(R_k)=\frac{\frac{1}{p}Z_k^T(zI-S)^{-2}Z_k}{1+\frac{1}{p}Z_k^T(zI-S)^{-1}Z_k}
\end{equation}
from
\[(A+uv^T)^{-1}=A^{-1}-\frac{A^{-1}uv^TA^{-1}}{1+v^TA^{-1}u}\]
(an identity established in Lemma~\ref{linalglemma}), and
\begin{equation}\label{res2}
    |\frac{\frac{1}{p}Z_k^T(zI-S)^{-2}Z_k}{1+\frac{1}{p}Z_k^T(zI-S)^{-1}Z_k}| \leq \frac{\frac{1}{p}Z_k^T((Re(z)I-S)^2+(Im(z))^2I)^{-2}Z_k}{Im(\frac{-1}{p}Z_k^T(zI-S)^{-1}Z_k)}=\frac{1}{Im(z)},
\end{equation}
a consequence of the triangle inequality,
\[\frac{1}{|z-\lambda|^2}=\frac{1}{(Re(z)-\lambda)^2+(Im(z))^2}=\frac{1}{Im(z)}Im(\frac{-1}{z-\lambda})\] 
as well as
\[\frac{1}{z-\lambda}=\frac{Re(z)-\lambda-iIm(z)}{(Re(z)-\lambda)^2+(Im(z))^2} \hspace{0.5cm} (\lambda \in \mathbb{R}).\]

\par
Return now to (\ref{mgdiff}): since \(|\delta_k| \leq \frac{2}{Im(z)}\) and they form martingale differences, an extension of the Burkholder's inequality for complex-valued martingales (Lemma \(2.12\) in Bai and Silverstein~\cite{baisilvbook}) provides
\begin{equation}\label{bd4th}
    \mathbb{E}[|m_{F_n}(z)-\mathbb{E}[m_{F_n}(z)]|^4] \leq \frac{K_4}{N^4} \cdot \mathbb{E}[(\sum_{1 \leq k \leq p}{|\delta_k|^2})^2] \leq \frac{K_4}{N^4} \cdot \frac{16p^2}{(Im(z))^4}=\frac{C(z,\gamma_n)}{N^2},
\end{equation}
for some universal \(K_4>0,\) from which claim (\ref{as}) ensues by the first Borel-Cantelli lemma, \(N \geq n,\) and Markov's inequality due to
\[\mathbb{P}(\sup_{r \geq M}{|m_{F_r}(z)-\mathbb{E}[m_{F_r}(z)]| \geq \epsilon}) \leq \sum_{r \geq M}{\frac{1+C(z,\gamma)}{n^2 \cdot \epsilon^4}} \hspace{1cm} (\epsilon>0, M \geq M_0).\]

\subsection*{Second Step}

Fix again \(z \in \mathbb{C}^+:\) (\ref{554}) holds, i.e.,
\[-1+zm_{F_n}(z)=-\frac{1}{\gamma_n}+\frac{1}{N}\sum_{1 \leq j \leq p}{\frac{1}{1-\frac{1}{p}Z_j^TR_jZ_j}}\]
and yields
\begin{equation}\label{eq1}
    -1+z\mathbb{E}[m_{F_n}(z)]=-\frac{1}{\gamma_n}+\frac{1}{\gamma_n}\mathbb{E}[\frac{1}{1-\frac{1}{p}Z_1^TR_1Z_1}]
\end{equation}
(\(\mathbb{E}[\frac{1}{|1-\frac{1}{p}Z_1^TR_1Z_1|}]<\infty\) is immediate from \(\frac{1}{|1-\frac{1}{p}Z_1^TR_1Z_1|} \leq 1+\frac{1}{Im(z)} \cdot \frac{1}{p}||Z_1||^2:\) see (\ref{eq1234})). When
\begin{equation}\label{simplifiedcond}
    Var(\frac{1}{p}Z_1^TR_1Z_1)=o(1),
\end{equation}
the random variable on the right-hand side of (\ref{eq1}) can be shown to be close to 
\[\frac{1}{1-\frac{1}{p}tr((zI-S))^{-1}}=\frac{1}{1-\gamma_nm_{F_n}(z)},\]
which in turn is roughly
\[\frac{1}{1-\mathbb{E}[\frac{1}{p}tr((zI-S))^{-1}]}=\frac{1}{1-\gamma_n\mathbb{E}[m_{F_n}(z)]}\] 
(via step \(1\)). The variance condition (\ref{simplifiedcond}) is crucial in Bai and Zhou~\cite{baizhou} and represents the gist of the argument behind (\ref{53}) (Bai and Zhou~\cite{baizhou} formulate it differently, \(Var(\frac{1}{p}Z_1^TBZ_1)=o(1)\) when \(B=B^T\) is deterministic with \(||B|| \leq 1:\) in the current isotropic case, (\ref{simplifiedcond}) would suffice). However, for the model considered in Theorem~\ref{mainth}, (\ref{simplifiedcond}) does not appear to hold in the entire range \(\frac{d}{n^{1/2}}=o(1)\) (Bryson et al.~\cite{bryson} justify (\ref{simplifiedcond}) when \(\frac{d}{n^{1/3}}=o(1)\)). 
\par
To bypass the aforesaid issue, a vector truncation is employed: this is effective so long as the length of \(||Z_1||\) is concentrated around its mean. Especially, this occurs when \(\mathbb{E}[x_{1}^4]=o(\frac{n}{d^2})\) by virtue of Lemma~\ref{lenlemma}, and it must be said that the bound on the variance is tight: see (\ref{lbvar}) as well as Lemma~\ref{lenlemma2} for what can occur outside the range \(\frac{d}{n^{1/2}}=o(1),\) implicit in the conditions of Theorem~\ref{mainth} (see comments below its statement). 
The conclusion for \(2.\) (i.e., step \(2\)) 
ensues from the following result.

\begin{proposition}\label{propstep2}
     Suppose the assumptions in Theorem~\ref{mainth} hold. 
     Then 
     \begin{equation}\tag{\(s2.0\)}\label{concl1}
         \mathbb{E}[\frac{1}{1-\frac{1}{p}Z_1^TR_1Z_1}]=\frac{1}{1-\gamma_n\mathbb{E}[m_{F_n}(z)]}+o(1)
     \end{equation}
     when \(Im(z) \geq 1+4\gamma.\) Particularly, if \(Im(z) \geq \frac{16}{9}+4\gamma,\) then
     \begin{equation}\tag{\(s2\)}\label{concl2}
         \mathbb{E}[m_{F_n}(z)]=m_{\gamma}(z)+o(1).
     \end{equation}
\end{proposition}

The proof of Proposition~\ref{propstep2}, contained in section~\ref{sectdetails}, consists of several steps: the starting point is writing
\[\mathbb{E}[\frac{1}{1-\frac{1}{p}Z_1^TR_1Z_1}]=\mathbb{E}[\frac{1}{1-\frac{1}{p}Z_1^TR_1Z_1}\cdot \chi_{||Z_1||^2 \leq 2N}]+\mathbb{E}[\frac{1}{1-\frac{1}{p}Z_1^TR_1Z_1}\cdot \chi_{||Z_1||^2>2N}].\]
The second term is shown to be negligible (see Lemma~\ref{tailev}), while the first is dealt with via a power series expansion at \(0.\) The constraint on \(||Z_1||\) and \(Im(z)>4\gamma\) ensure \(\frac{1}{p}|Z_1^TR_1Z_1| \leq \frac{1}{2}\) for \(n \geq n_0\) (see Lemma~\ref{lemmasmall}), and this in conjunction with a truncation of
\[\frac{1}{1-x}=\sum_{k \geq 0}{x^k}\]
allows computing the size of the first expectation above: the \(q^{th}\) partial sum is employed instead of the power series for some \(q=q(n)\to \infty\) (recall the comment preceding Lemma~\ref{lemmaequiv}) and the \(k^{th}\) term is shown to be close to \((\mathbb{E}[\frac{1}{p}tr(R_1)])^k,\) which is in turn close to \((\gamma_n\mathbb{E}[m_{F_n}(z)])^k.\) These two results will produce (\ref{concl1}), while (\ref{concl2}) will be inferred in a similar vein to how (\ref{53}) was deduced in \cite{oldpaper} (see also (\ref{zeq})). 

\subsection*{Third Step}

Suppose the assumptions in Theorem~\ref{mainth} hold. Proposition~\ref{propstep2} yields
\[\mathbb{E}[m_{F_n}(z)]=m_{\gamma}(z)+o(1)\]
when \(z \in D(\gamma)\) (recall (\ref{dgamma})) and gives
together with (\ref{as}) that 
\[m_{F_n}(z)-m_{\gamma}(z) \xrightarrow[]{a.s.} 0.\]
Let \(S(\gamma)=\{q_1+iq_2:q_1,q_2 \in \mathbb{Q}, q_2 \geq \frac{16}{9}+4\gamma\},\) a countable set of points in \(D(\gamma)\) dense in it: then
\[m_{F_n}(z)-m_{\gamma}(z) \to 0\]
for all \(z \in S(\gamma)\) almost surely. This yields
\begin{equation}\tag{\(s3\)}\label{s3}
    m_{F_n}(z)-m_{\gamma}(z) \to 0
\end{equation}
for all \(z \in D(\gamma)\) almost surely since \(m_{\gamma}\) is clearly continuous in \(\mathbb{C}^{+},\) and
\[|m_{F_n}(z_1)-m_{F_n}(z_2)| \leq \frac{|z_1-z_2|}{(\min{(Im(z_1),Im(z_2))})^2}\]
from 
\[\frac{1}{z_1-\lambda}-\frac{1}{z_2-\lambda}=\frac{z_2-z_1}{(z_1-\lambda)(z_2-\lambda)} \hspace{0.5cm} (\lambda \in \mathbb{R}).\]
\par
Take any subsequence \(n_k\) such that \(F_{n_k}\) converges weakly to some distribution \(G\) (no assumption made on \(G:\) in particular, \(G(\infty)>0\) could occur). The Stieltjes transforms must also converge pointwise with probability one (for each \(z \in \mathbb{C}^+,\) their real and imaginary parts are real-valued continuous functions bounded by \(\frac{1}{Im(z)}\)), and so \(G\) has the same Stieltjes transform as \(F_{\gamma},\) the Marchenko-Pastur law with parameter \(\gamma,\) in \(D(\gamma)\) from (\ref{s3}). Since the Stieltjes transform of \(G\) is analytical in \(\mathbb{C}^{+}\) from
\[\frac{m(z_1)-m(z_2)}{z_1-z_2}=-\int{\frac{1}{(z_1-\lambda)(z_2-\lambda)}dG(\lambda)}, \hspace{0.5cm} \lim_{z \to z_1}{\frac{1}{(z_1-\lambda)(z-\lambda)}}= \frac{1}{(z_1-\lambda)^2} \hspace{0.3cm} (z_1,z_2 \in \mathbb{C}^{+}),\]
and 
the dominated convergence theorem, it follows that \(m_G\) and \(m_{\gamma}\) are analytical and equal in \(D(\gamma) \subset \mathbb{C}^{+},\) whereby any of their analytical continuations coincide, from which \(m_G=m_\gamma.\) 
This implies \(G=F_{\gamma}\) (defined by (\ref{mplaws})) 
and concludes Theorem~\ref{mainth} since all weakly convergent subsequences 
 have limit \(F_{\gamma}.\)
\par
This completes the discussion on the three steps employed for justifying Theorem~\ref{mainth}.

\section{Expectations via Truncations}\label{sectdetails}

The goal of this section is proving Proposition~\ref{propstep2}: assume throughout this section that the conditions stated in Proposition~\ref{propstep2} hold. Two properties used below repeatedly are 
\begin{equation}\label{condd}
    \frac{d}{n^{1/2}}=o(1), \hspace{0.5cm} C^{q} \cdot \frac{d}{n^{1/2}}=o(1),
\end{equation}
for some \(q=q(n) \in \mathbb{N}, q(n) \to \infty\) (the latter is entailed by Lemma~\ref{lemmaequiv} applied to \(a_n=C(n),b_n=\frac{n^{1/2}}{d},\) with the former a direct consequence of the latter and \(C \geq 2\)).
\par
Begin by showing (\ref{concl2}) (contingent on (\ref{concl1})). Identities (\ref{eq1}) and (\ref{concl1}) give
\[-1+z\mathbb{E}[m_{F_n}(z)]=-\frac{1}{\gamma_n}+\frac{1}{\gamma_n}\mathbb{E}[\frac{1}{1-\frac{1}{p}Z_1^TR_1Z_1}]=-\frac{1}{\gamma_n}+\frac{1}{\gamma_n} \cdot \frac{1}{1-\gamma_n\mathbb{E}[m_{F_n}(z)]}+o(1).\]
Recall (\ref{55}), i.e., 
\[-1+zm_n(z)=-\frac{1}{\gamma_n}+\frac{1}{\gamma_n} \cdot \frac{1}{1-\gamma_n m_n(z)},\]
from which
\[z(\mathbb{E}[m_{F_n}(z)]-m_n(z))=\frac{1}{\gamma_n} \cdot (\frac{1}{1-\gamma_n \mathbb{E}[m_{F_n}(z)]}-\frac{1}{1-\gamma_n m_n(z)})+o(1),\]
or equivalently,
\begin{equation}\label{o1id}
    (\mathbb{E}[m_{F_n}(z)]-m_n(z)) \cdot (z-\frac{1}{(1-\gamma_n m_{n}(z))(1-\gamma_n \mathbb{E}[m_{F_n}(z)])})=o(1).
\end{equation}
When \(n \geq n_0,\) \(z\) has \(Im(z) \geq 4\gamma_n,\) from which 
\[|z-\frac{1}{(1-\gamma_n m_{n}(z))(1-\gamma_n \mathbb{E}[m_{F_n}(z)])}| \geq |z|-\frac{1}{|1-\gamma_n m_{n}(z)| \cdot |1-\gamma_n \mathbb{E}[m_{F_n}(z)]|} \geq |z|-\frac{1}{(1-\gamma_n \cdot \frac{1}{Im(z)})^2} \geq 4\gamma\]
via the triangle inequality and
\[|z|-\frac{1}{(1-\gamma_n \cdot \frac{1}{Im(z)})^2} \geq |z|-\frac{1}{(1-\gamma_n \cdot \frac{1}{4\gamma_n})^2} \geq Im(z)-\frac{16}{9}.\]
So (\ref{o1id}) entails
\[\mathbb{E}[m_{F_n}(z)]-m_n(z)=o(1).\]
and \(\gamma_n \to \gamma\) gives \(m_n(z) \to m_\gamma(z),\) whereby (\ref{concl2}) follows, i.e.,
\[\mathbb{E}[m_{F_n}(z)]-m_\gamma(z)=o(1).\]
\par
Having completed the justification of (\ref{concl2}), return to (\ref{concl1}), the objective of the remainder of this section. Start with
\[\mathbb{E}[\frac{1}{1-\frac{1}{p}Z_1^TR_1Z_1}]=\mathbb{E}[\frac{1}{1-\frac{1}{p}Z_1^TR_1Z_1}\cdot \chi_{||Z_1||^2>2N}]+\mathbb{E}[\frac{1}{1-\frac{1}{p}Z_1^TR_1Z_1}\cdot \chi_{||Z_1||^2 \leq 2N}].\]
Subsection~\ref{subsect3.1} analyzes the first expectation, while subsections~\ref{subsect2}-\ref{subsect4} 
treat the second. The conclusion of Proposition~\ref{propstep2} results from (\ref{finid1}) and (\ref{bulk1})-(\ref{bulk2}) proved therein.

\subsection{Tail Expectation}\label{subsect3.1}

This subsection justifies that under the conditions in Proposition~\ref{propstep2},
\begin{equation}\label{finid1}\tag{\(E1\)}
    \mathbb{E}[\frac{1}{1-\frac{1}{p}Z_1^TR_1Z_1}\cdot \chi_{||Z_1||^2>2N}]=o(1).
\end{equation}
This is accomplished via the following lemma: since \(\frac{d}{n^{1/2}}=o(1), \mathbb{E}[x^4] \leq 4C^2 \leq \frac{4n^{1/2}}{d}\) (recall (\ref{condd}) and \(C \geq 2\)), the conditions in Lemma~\ref{tailev} below hold for \(n\) sufficiently large, whereby (\ref{finid1}) ensues: the bound yielded by this result is \(O(\frac{d}{n^{1/2}})=o(1).\)

\begin{lemma}\label{tailev}
  Suppose \(d \leq \frac{n}{2},\) and \(\mathbb{E}[x^4] \leq \frac{n-2d+2}{(d-1)^2}.\) Then
    \[|\mathbb{E}[\frac{1}{1-\frac{1}{p}Z_1^TR_1Z_1} \cdot \chi_{||Z_1||^2>2N}]| \leq (1+\frac{2\gamma_n}{Im(z)}) \cdot 2\mathbb{E}[x^4] \cdot \frac{d^2}{n}.\]
\end{lemma}

\begin{proof}
    Identity (\ref{eq123}) can be rearranged as
\begin{equation}\label{eq1234}
    \frac{1}{1-\frac{1}{p}Z_j^TR_jZ_j}=1+\frac{1}{p}Z_j^T(zI-S)^{-1}Z_j,
\end{equation}
entailing the quantity of interest is at most
\begin{equation}\label{bdnorm}
    \mathbb{E}[(1+\frac{1}{p}|Z_1^T(zI-S)^{-1}Z_1|) \cdot \chi_{||Z_1||^2>2N}] \leq \mathbb{E}[(1+\frac{1}{Im(z)} \cdot \frac{1}{p}||Z_1||^2) \cdot \chi_{||Z_1||^2>2N}]
\end{equation}
by using \(||(zI-S)^{-1}|| \leq \frac{1}{Im(z)}.\) This last bound is
\[\mathbb{E}[(1+\frac{1}{Im(z)} \cdot \frac{1}{p}||Z_1||^2) \cdot \chi_{||Z_1||^2>2N}]\leq (1+\frac{1}{Im(z)} \cdot \frac{N}{p}) \cdot \mathbb{P}(||Z_1||^2>2N)+\frac{1}{Im(z)} \cdot \frac{N}{p} \cdot Var(\frac{1}{N}||Z_1||^2)\]
since \(\mathbb{E}[\frac{1}{N}||Z_1||^2]=1\) gives
\[\mathbb{E}[(\frac{1}{N}||Z_1||^2-1) \cdot \chi_{||Z_1||^2>2N}] \leq \mathbb{E}[(\frac{1}{N}||Z_1||^2-1)^2]=Var(\frac{1}{N}||Z_1||^2).\]
Chebyshev's inequality gives 
\[\mathbb{P}(||Z_1||^2>2N)=\mathbb{P}(|\hspace{0.05cm} ||Z_1||^2-N|>N) \leq \frac{Var(||Z_1||^2)}{N^2},\] 
and together with Lemma~\ref{lenlemma}, it completes the proof insofar as
\[(1+\frac{1}{Im(z)} \cdot \frac{N}{p}) \cdot \mathbb{P}(||Z_1||^2>2N)+\frac{1}{Im(z)} \cdot \frac{N}{p} \cdot Var(\frac{1}{N}||Z_1||^2) \leq (1+\frac{2\gamma_n}{Im(z)}) \cdot Var(\frac{1}{N}||Z_1||^2) \leq \]
\begin{equation}\label{finalineq}
     \leq (1+\frac{2\gamma_n}{Im(z)}) \cdot 2\mathbb{E}[x^4] \cdot \frac{d^2}{n}.
\end{equation} 
\end{proof}

\(\newline\)
\textbf{Remark:} The bound ensuing from (\ref{eq1234}) could be made uniform,
\[\frac{1}{|1-\frac{1}{p}Z_j^TR_jZ_j|} \leq \frac{|z|}{Im(z)}\]
(see section~\ref{sectgen}: particularly, (\ref{prodinv})): however, concentration of the length of \(Z_1\) (captured by \(\mathbb{P}(||Z_1||^2>2N)\)) continues to be needed to guarantee (\ref{finid1}).


\subsection{Power Series Expansion}\label{subsect2}

In this subsection and the forthcoming one, the object of interest is
\[\mathbb{E}[\frac{1}{1-\frac{1}{p}Z_1^TR_1Z_1}\cdot \chi_{||Z_1||^2 \leq 2N}].\]
Since \(1+4\gamma \geq 4 \gamma_n\) for \(n\) sufficiently large, the following result entails that for \(n \geq n_0,\) 
\begin{equation}\label{smalllen}
    \frac{1}{p}|Z_1^TR_1Z_1| \cdot \chi_{||Z_1||^2 \leq 2N} \leq \frac{1}{2}.
\end{equation}

\begin{lemma}\label{lemmasmall}
    If \(Im(z) \geq 4\gamma_n,\) then \(\frac{1}{p}|Z_1^TR_1Z_1| \cdot \chi_{||Z_1||^2 \leq 2N} \leq \frac{1}{2}.\)
\end{lemma}

\begin{proof}
    Since 
    \[\frac{1}{|z-\lambda|} \leq \frac{1}{Im(z)} \hspace{0.5cm} (\lambda \in \mathbb{R}),\]
    it follows that \(||R_1||=||(zI-\frac{1}{p}\sum_{2 \leq j \leq p}{Z_jZ_j^T})^{-1}|| \leq \frac{1}{Im(z)},\) whereby
    \[\frac{1}{p}|Z_1^TR_1Z_1|\cdot \chi_{||Z_1||^2 \leq 2N} \leq \frac{1}{Im(z)} \cdot \frac{1}{p}||Z_1||^2 \cdot  \chi_{||Z_1||^2 \leq 2N} \leq \frac{2\gamma_n}{Im(z)} \leq \frac{1}{2}.\]
\end{proof}

For \(n \in \mathbb{N},\)
\[M=M(n)=\lfloor \min{(\frac{q(n)}{12},\frac{\log{(1+\frac{n^{1/2}}{d})}}{(\log{\log{(1+\frac{n^{1/2}}{d})}})^2})} \rfloor,\]
with \(q=q(n) \in \mathbb{N}, q(n) \to \infty\) satisfying (\ref{condd}).
The key properties of this sequence are
\begin{equation}\label{propm}
    M(n) \in \mathbb{Z}_{\geq 0},  \hspace{0.8cm} M(n) \to \infty, \hspace{0.8cm} (16CM(n))^{M(n)} \cdot (\frac{d}{n^{1/2}})^{1/8} \to 0:
\end{equation}
the first two are immediate from the definitions of \(M(n),q(n),\) and (\ref{condd}), while the last ensues from
\[(16CM(n))^{M(n)} \cdot (\frac{d}{n^{1/2}})^{1/8}=C^{M(n)}\cdot (\frac{d}{n^{1/2}})^{1/12} \cdot (16M(n))^{M(n)} \cdot (\frac{d}{n^{1/2}})^{1/24} \to 0 \cdot 0=0\]
by using anew (\ref{condd}).
For ease of notation, the dependence on \(n\) is dropped, i.e., \(M(n)\) is denoted by \(M.\) By virtue of (\ref{smalllen}), the identity
\begin{equation}\label{powser}
    \frac{1}{1-x}=1+x+...+x^{M}+\frac{x^{M+1}}{1-x} \hspace{0.5cm} (x \ne 1)
\end{equation}
gives
\begin{equation}\tag{\(E2a\)}\label{bulk1}
    \mathbb{E}[\frac{1}{1-\frac{1}{p}Z_1^TR_1Z_1} \cdot \chi_{||Z_1||^2 \leq 2N}]=1+\Sigma_1+\Sigma_2-\Sigma_3,
\end{equation}
\[\Sigma_1=\sum_{1 \leq m \leq M}{\mathbb{E}[(\frac{1}{p}Z_1^TR_1Z_1)^m]}, \hspace{0.2cm} \Sigma_2=\mathbb{E}[\frac{(\frac{1}{p}Z_1^TR_1Z_1)^{M+1}}{1-\frac{1}{p}Z_1^TR_1Z_1} \cdot \chi_{||Z_1||^2 \leq 2N}], \hspace{0.2cm} \Sigma_3=\sum_{0 \leq m \leq M}{\mathbb{E}[(\frac{1}{p}Z_1^TR_1Z_1)^m\cdot \chi_{||Z_1||^2>2N}]}.\]
Each of these components is dealt with separately: namely, it is shown that
\begin{equation}\tag{\(E2b\)}\label{bulk2}
    \Sigma_1=-1+\frac{1}{1-\gamma_n\mathbb{E}[m_{F_n}(z)]}+o(1), \hspace{0.5cm} \Sigma_2=o(1), \hspace{0.5cm} \Sigma_3=o(1).
\end{equation}
The last two claims ensue from (\ref{iibound}) and (\ref{iiibound}) below since \(M=M(n) \to \infty,\) whereas the first is tackled in the next subsection.
\par
\vspace{0.2cm}
For \(\Sigma_2,\) (\ref{eq1234}) gives
\[\frac{(\frac{1}{p}Z_1^TR_1Z_1)^{M+1}}{1-\frac{1}{p}Z_1^TR_1Z_1}=\frac{1}{p}Z_1^T(zI-S)^{-1}Z_1 \cdot (\frac{1}{p}Z_1^TR_1Z_1)^{M},\]
whereby the pointwise bound \(\frac{1}{p}|Z_1^TBZ_1| \leq \frac{1}{p}||Z_1||^2 \cdot ||B||\) for any Hermitian matrix \(B\) yields that for \(n \geq n_0,\)
\begin{equation}\label{iibound}\tag{\(e_2\)}
    |\Sigma_2| \leq (\frac{2\gamma_n}{Im(z)})^{M+1} \leq (\frac{\frac{1}{2}+2\gamma}{1+4\gamma})^{M+1}=2^{-M-1}.
\end{equation}
\par
For \(\Sigma_3,\) Lemma~\ref{lemmamom} can be applied to
\[d_1=d_2=...=d_m=d, \hspace{0.7cm} m \leq 2M, \hspace{0.7cm} n \geq n_1\] 
because \(md \leq \frac{n^{1/2}}{3Ce}\) when \(n_1\) is large enough to ensure \(C^2 \cdot \frac{d}{n^{1/2}} \leq \frac{1}{9e^2}, M^2 \leq \frac{n^{1/2}}{4d}\) (recall (\ref{condd}) and (\ref{propm})) insofar as in these situations, 
\[md \leq \frac{n^{1/4}}{d^{1/2}} \cdot d=n^{1/4} \cdot d^{1/2} \leq \frac{n^{1/2}}{3Ce}.\] 
This result entails that in such cases
\begin{equation}\label{bddmomm}
    \mathbb{E}[(\frac{1}{N}||Z_1||^2)^{m}] \leq 1+e^{\frac{2C^2e^2m^2d^2}{n}} \cdot \frac{2C^2e^2m^2d^2}{n},
\end{equation}
which in turn yields, similarly to the argument for Lemma~\ref{tailev}, that for \(n \geq n_0+n_1,\)
\begin{equation}\label{iiibound}\tag{\(e_3\)}
    |\Sigma_3| \leq \frac{2(M+1)}{2^M}
\end{equation}
from
\[|\Sigma_3|=|\sum_{0 \leq m \leq M}{\mathbb{E}[(\frac{1}{p}Z_1^TR_1Z_1)^m\cdot \chi_{||Z_1||^2>2N}]}| \leq \sum_{0 \leq m \leq M}{\mathbb{E}[(\frac{1}{p}||Z_1||^2)^m \cdot \frac{1}{(Im(z))^m} \cdot \chi_{||Z_1||^2>2N}]} \leq\]
\begin{equation}\label{sigma3bd}
    \leq\sum_{0 \leq m \leq M}{\frac{\gamma_n^{m}}{(Im(z))^m} \cdot \frac{2}{2^M}} \leq \frac{2(M+1)}{2^M}.
\end{equation}
The last inequality can be justified as follows: \(n \geq n_0\) guarantees \(\gamma_n \leq \frac{1}{4}+\gamma<1+4\gamma \leq Im(z),\) which alongside (\ref{bddmomm}) gives that for \(0 \leq m \leq M,\)
\[\mathbb{E}[(\frac{1}{N}||Z_1||^2)^m \cdot \chi_{\frac{1}{N}||Z_1||^2>2}] \leq \frac{1}{2^{M}} \cdot \mathbb{E}[(\frac{1}{N}||Z_1||^2)^{m+{M}}] \leq \frac{1+e^{2/9} \cdot 2/9}{2^{M}} \leq \frac{2}{2^{M}}\]
as \(e^{2/9}<e<3<\frac{9}{2}.\)

\subsection{High Moments}\label{subsect3.3}

The goal of this subsection is showing
\[\Sigma_1=\sum_{1 \leq m \leq M}{\mathbb{E}[(\frac{1}{p}Z_1^TR_1Z_1)^m]}=-1+\frac{1}{1-\gamma_n\mathbb{E}[m_{F_n}(z)]}+o(1),\]
which will conclude the proof of Proposition~\ref{propstep2}. Begin by noting that it suffices to prove that
\begin{equation}\label{r1eq}
    \Sigma_1=\sum_{1 \leq m \leq M}{\mathbb{E}[(\frac{1}{p}Z_1^TR_1Z_1)^m]}=-1+\frac{1}{1-\mathbb{E}[\frac{1}{p}tr(R_1)]}+o(1):
\end{equation}
(\ref{res1}) and (\ref{res2}) provide
\[|tr((zI-S)^{-1})-tr(R_1)|=|\frac{\frac{1}{p}Z_1^T(zI-S)^{-2}Z_1}{1+\frac{1}{p}Z_1^T(zI-S)^{-1}Z_1}| \leq \frac{1}{Im(z)},\]
whereby for \(n \geq n_0,\)
\[|\frac{1}{1-\gamma_n\mathbb{E}[m_{F_n}(z)]}-\frac{1}{1-\mathbb{E}[\frac{1}{p}tr(R_1)]}| \leq \frac{\mathbb{E}[|\frac{1}{p}tr((zI-S)^{-1})-\frac{1}{p}tr(R_1)|]}{(1-\frac{\gamma_n}{Im(z)})^2} \leq \frac{\frac{1}{p} \cdot \frac{1}{Im(z)}}{(1-\frac{1}{4})^2}\]
by using \(\frac{1}{N}|tr(R_1)|,\frac{1}{N}|tr((zI-S)^{-1})| \leq \frac{1}{Im(z)},\) and so
\[\frac{1}{1-\gamma_n\mathbb{E}[m_{F_n}(z)]}=\frac{1}{1-\mathbb{E}[\frac{1}{p}tr(R_1)]}+o(1).\]
\par
Fix \(1 \leq m \leq M,\) and consider
\begin{equation}\label{sumexp}
    \mathbb{E}[(\frac{1}{p}Z_1^TR_1Z_1)^m]=\gamma_n^m \cdot N^{-m}\sum_{1 \leq i_1,i_2,\hspace{0.05cm}...\hspace{0.05cm},i_{2m} \leq N}{\mathbb{E}[R_{i_1i_2}R_{i_3i_4}...R_{i_{2m-1}i_{2m}}] \cdot \mathbb{E}[Z_{1i_1}Z_{1i_2}...Z_{1i_{2m}}]},
\end{equation}
where for ease of notation \(R_{kl}:=(R_1)_{kl},\) \(1 \leq k,l \leq N\) (the equality follows by conditioning on \(\sigma(Z_{j},j>1),\) a \(\sigma\)-algebra independent of \(\sigma(Z_1),\) and the tower property). It is shown next that
\begin{equation}\label{indbd}
    \mathbb{E}[(\frac{1}{N}Z_1^TR_1Z_1)^m]=\mathbb{E}[(\frac{1}{N}tr(R_1))^m]+(Im(z))^{-m} \cdot [O((16Cm)^{m} \cdot (\frac{d}{n^{1/2}})^{1/3})+O(m^2 \cdot \frac{d^2}{n})]:
\end{equation}
this suffices for the derivation of (\ref{r1eq}) because it will render
\[\Sigma_1=\sum_{1 \leq m \leq M}{\mathbb{E}[(\frac{1}{p}tr(R_1))^m]}+\sum_{1 \leq m \leq M}{(\frac{\gamma_n}{Im(z)})^m \cdot [O((16Cm)^{m} \cdot (\frac{d}{n^{1/2}})^{1/3})+O(m^2 \cdot \frac{d^2}{n})]}=\]
\[=-1+\mathbb{E}[\frac{1}{1-\frac{1}{p}tr(R_1)}]-\mathbb{E}[\frac{(\frac{1}{p}tr(R_1))^{M+1}}{1-\frac{1}{p}tr(R_1)}]+[O((16CM)^{M} \cdot (\frac{d}{n^{1/2}})^{1/3})+O(\frac{d^2}{n})]\]
via \(\frac{1}{p}|tr(R_1)| \leq \gamma_n \cdot ||R_1|| \leq \frac{\gamma_n}{1+4\gamma} \leq \frac{1}{4}\) for \(n \geq n_0,\) \(M \to \infty,\) identity (\ref{powser}), and \(\frac{(m+1)^2}{m^2} \leq \frac{9}{4}<4\) for \(m \geq 2.\) The last two terms are \(o(1)\) by the definition of \(M\) and \(\frac{d}{n^{1/2}}=o(1)\) (see (\ref{propm}) and (\ref{condd})), rendering 
(\ref{r1eq}), 
\[\Sigma_1=-1+\mathbb{E}[\frac{1}{1-\frac{1}{p}tr(R_1)}]+o(1)=-1+\frac{1}{1-\mathbb{E}[\frac{1}{p}tr(R_1)]}+o(1)\]
insomuch as
\[|\mathbb{E}[\frac{1}{1-\frac{1}{p}tr(R_1)}]-\frac{1}{1-\mathbb{E}[\frac{1}{p}tr(R_1)]}| \leq \mathbb{E}[|\frac{1}{1-\frac{1}{p}tr(R_1)}-\frac{1}{1-\mathbb{E}[\frac{1}{p}tr(R_1)]}|] \leq \frac{\mathbb{E}[|\frac{1}{p}tr(R_1)-\mathbb{E}[\frac{1}{p}tr(R_1)]|]}{(1-\frac{1}{4})^2}=o(1)\]
for \(n \geq n_0,\) the last claim following from (\ref{bd4th}) and Hölder's inequality via \(\frac{1}{p}tr(R_1) \overset{d}{=} \frac{p-1}{p} \cdot m_{F^*_{n}}(z),\) where \(F^*_{n}=F_{\frac{1}{p-1}\sum_{1 \leq j \leq p-1}{Z_jZ_j^T}}\) is the empirical spectral distribution corresponding to \(p-1\) i.i.d. samples, or alternatively, by using the dominated convergence theorem, \(\frac{1}{p}|tr(R_1)| \leq \frac{1}{4},\) and (\ref{as}).
\vspace{0.2cm}
\par
It remains to justify (\ref{indbd}). Some terminology that captures the dependencies among the factors underlying the expectations in it is in order. Call \(2l-1,2l\) \textit{counterparts} when \(l \in \mathbb{N},\) while for \(1 \leq j_1,j_2 \leq N,\) say \(j_1,j_2\) are \textit{adjacent} if under the (fixed) bijection \(b_{d,n}\) used to label the entries of \(Z_1\) with the elements of
\[\{(r_1,r_2,\hspace{0.05cm}...\hspace{0.05cm},r_d): 1 \leq r_1<r_2<...<r_d \leq n\},\] 
\(j_1\) and \(j_2\) share some index, i.e.,
\[b_{d,n}(j_1)=(r_1,r_2,\hspace{0.05cm}...\hspace{0.05cm},r_d), \hspace{0.3cm} b_{d,n}(j_2)=(r'_1,r'_2,\hspace{0.05cm}...\hspace{0.05cm},r'_d), \hspace{0.3cm} \{r_1,r_2,\hspace{0.05cm}...\hspace{0.05cm},r_d\} \cap \{r'_1,r'_2,\hspace{0.05cm}...\hspace{0.05cm},r'_d\} \ne \emptyset.\] 
Lastly, for a tuple \((i_j)_{1 \leq j \leq 2m}\) and \(1 \leq v \leq 2m,\) let \(shdeg(i_v)\) (the shared degree of \(i_v\)) be the number of positions in the tuple underlying \(i_v\) that appear in some \(i_w\) for \(w \ne v\) and \(w\) not the counterpart of \(v:\) i.e., when \(b_{d,n}(i_w)=(r_{wj})_{1 \leq j \leq d},\) the shared degrees are given by
\[shdeg(i_v)=|\{r_{v1},r_{v2},\hspace{0.05cm}...\hspace{
0.05cm},r_{vd}\} \cap \{r_{y1},r_{y2},\hspace{0.05cm}...\hspace{
0.05cm},r_{yd}: 1 \leq y \leq 2m, y \not \in S(v)\}| \hspace{0.5cm} (1 \leq v \leq 2m),\]
where \(S(t)=\{2 \lceil \frac{t+1}{2} \rceil-1, 2 \lceil \frac{t+1}{2} \rceil\}\) (i.e., \(S(t)\) is the set consisting of \(t \in \mathbb{N}\) and its counterpart). Intuitively the first term in (\ref{indbd}) is given by tuples \((i_j)_{1 \leq j \leq 2m}\) whose shared degrees vanish and in which entries corresponding to counterparts are equal, i.e., 
\[i_{2l-1}=i_{2l} \hspace{0.2cm} (1 \leq l \leq m), \hspace{0.5cm} shdeg(i_v)=0 \hspace{0.2cm} (1 \leq v \leq 2m).\]
The former gives rise to products of diagonal entries of \(R_1,\) while the latter implies \((Z_{1i_{2l-1}})_{1 \leq l \leq m}\) are independent, whereby 
\[\mathbb{E}[Z_{1i_1}Z_{1i_2}...Z_{1i_{2m}}]=\prod_{1 \leq l \leq m}{\mathbb{E}[Z^2_{1i_{2l-1}}]}=1.\] 
The argument below shows this is indeed the case: notice the contribution of such configurations is not exactly \(\mathbb{E}[(\frac{1}{N}tr(R_1))^m]\) due to the second constraint, \(shdeg(i_v)=0\) for all \(1 \leq v \leq 2m.\) 
\vspace{0.3cm}
\par 
Return to the sum in (\ref{indbd}). A change of summation is used, the key parameters being the shared degrees \((shdeg(i_v))_{1 \leq v \leq 2m}.\) The tuples with \(shdeg(i_{2l-1}) \ne shdeg(i_{2l})\) for some \(1 \leq l \leq m\) have vanishing contributions: \(Z_{1i_1}Z_{1i_2}...Z_{1i_{2m}},\) when written as a product of \(2md\) factors among \(x_1,x_2,\hspace{0.05cm}...\hspace{0.05cm},x_n,\) will have a monomial of degree \(1\) from a position \(s \in \{1,2,\hspace{0.05cm}...\hspace{0.05cm},n\}\) such that \(s\) that appears exactly in one of
the complements of the sets underlying \(shdeg(i_{2l-1}),shdeg(i_{2l})\) with respect to the sets formed by the entries of \(b_{d,n}(i_{2l-1}),b_{d,n}(i_{2l}),\) respectively,
and so it suffices to consider solely those with \(shdeg(i_{2l-1})= shdeg(i_{2l})\) for all \(1 \leq l \leq m.\) This entails that \(k=\frac{1}{2}\sum_{1 \leq v \leq 2m}{shdeg(i_v)} \in \mathbb{Z}_{\geq 0},\) and the number of such tuples is at most 
\[|\{(y_1,y_2,\hspace{0.05cm}...\hspace{0.05cm},y_m) \in \mathbb{Z}_{\geq 0}^m:y_1+...+y_m=k\}|=\binom{k+m-1}{m-1}.\] 
Tuples with fixed shared degrees \((shdeg(i_v))_{1 \leq v \leq 2m}\) contribute to the sum of interest
\[N^{-m}\sum_{1 \leq i_1,i_2,\hspace{0.05cm}...\hspace{0.05cm},i_{2m} \leq N}{\mathbb{E}[R_{i_1i_2}R_{i_3i_4}...R_{i_{2m-1}i_{2m}}] \cdot \mathbb{E}[Z_{1i_1}Z_{1i_2}...Z_{1i_{2m}}]}\]
at most
\begin{equation}\tag{\(fshdeg\)}\label{fixedcontr}
    (Im(z))^{-m} \cdot (\frac{d^2 \cdot 2^{14m}}{n})^{\frac{1}{12}\sum_{1 \leq v \leq 2m}{shdeg(i_v)}} \cdot (\mathbb{E}[x^{2m}])^{\frac{1}{2}\sum_{1 \leq v \leq 2m}{shdeg(i_v)}}.
\end{equation}
This identity is justified in the next subsection. Before delving into its proof, conclude the desired bound,~(\ref{indbd}), using it. 
\par
Take first the tuples with \(\sum_{1 \leq w \leq 2m}{shdeg(i_w)}>0:\) (\ref{fixedcontr}) gives that their contribution is at most
\[\sum_{2|k, k \geq 2}{\binom{k/2+m-1}{m-1} \cdot (Im(z))^{-m} \cdot (\frac{d^2 \cdot 2^{14m}}{n})^{\frac{k}{12}} \cdot (\mathbb{E}[x^{2m}])^{\frac{k}{2}}} \leq\]
\[\leq (Im(z))^{-m} \sum_{2|k, k \geq 2}{\binom{k+m-1}{m-1} \cdot (\frac{d^2 \cdot 2^{14m} \cdot (Cm)^{6m}}{n})^{\frac{k}{12}}} \leq\]
\begin{equation}\tag{\(T1\)}\label{t1}
    \leq 2^{m-1}(Im(z))^{-m}\sum_{2|k, k \geq 2}{(\frac{d^2 \cdot 2^{14m+12} \cdot (Cm)^{6m}}{n})^{\frac{k}{12}}}=O(2^{m}(Im(z))^{-m} \cdot (\frac{d^2 \cdot 2^{14m} \cdot (Cm)^{6m}}{n})^{\frac{1}{6}})
\end{equation}
via \(\binom{k+m-1}{m-1} \leq 2^{k+m-1},\) and for \(n \geq n_2,\)
\begin{equation}\label{small}
    \frac{d^2 \cdot 2^{14m+12} \cdot (Cm)^{6m}}{n} <\frac{2^{12} \cdot d^2 \cdot (8CM)^{6M}}{n} \leq  2^{12} \cdot (\frac{d^2}{n})^{1-3/4} \leq \frac{1}{2}
\end{equation}
by (\ref{propm}) and (\ref{condd}). 
\par
Consider now the tuples with \(\sum_{1 \leq w \leq 2m}{shdeg(i_w)}=0,\) and \(\mathbb{E}[Z_{1i_1}Z_{1i_2}...Z_{1i_{2m}}] \ne 0.\) In light of the discussion above, this is tantamount to \(i_{2l-1}=i_{2l},\) and \(i_{2j},i_{2j'}\) not adjacent for any \(1 \leq j<j'\leq m.\) This entails \(\mathbb{E}[Z_{1i_1}Z_{1i_2}...Z_{1i_{2m}}]=1,\) and what is left is
\[N^{-m}\sum_{(*),1 \leq j_1,j_2,\hspace{0.05cm}...\hspace{0.05cm},j_{m} \leq N}{\mathbb{E}[R_{j_1j_1}R_{j_2j_2}...R_{j_{m-1}j_{m}}]},\]
where \((*)\) denotes the additional constraint that \(j_k,j_l\) are not adjacent for any \(k \neq l.\) 
Take a tuple for which this last condition is violated: there must exist \(1 \leq k_1<k_2 \leq m\) such that \(j_{k_1},j_{k_2}\) are adjacent. This yields a contribution that in absolute value is at most
\[\binom{m}{2} \cdot \mathbb{E}[(\frac{1}{N}\sum_{1 \leq j \leq N}{|R_{jj}|})^{m-2} \cdot N^{-2}\sum_{(**),k_1,k_2}{|R_{k_1k_1}R_{k_2k_2}|}] \leq \binom{m}{2} \cdot (Im(z))^{-(m-2)} \cdot (Im(z))^{-2} \cdot \frac{d^2}{n},\]
where \((**)\) denotes the condition of \(k_1,k_2\) being adjacent, via
\[N^{-2}\sum_{(**),k_1,k_2}{|R_{k_1k_1}R_{k_2k_2}|} \leq N^{-1} \cdot (Im(z))^{-1}\max_{1 \leq k_1 \leq N}{\sum_{(**),k_2}{|R_{k_2k_2}|}} \leq\]
\[\leq N^{-1} \cdot (Im(z))^{-1} \cdot \max_{1 \leq k_1 \leq N}{\sqrt{N \cdot \frac{d^2}{n} \cdot \sum_{(**),k_2}{|R_{k_2k_2}|^2}}} \leq N^{-1} \cdot (Im(z))^{-1} \cdot \sqrt{(N \cdot \frac{d^2}{n})^2 \cdot (Im(z))^{-2}}=(Im(z))^{-2} \cdot \frac{d^2}{n}\]
(for any fixed \(k_1,\) there are at most \(d \cdot \binom{n-1}{d-1}=N \cdot \frac{d^2}{n}\) values \(k_2\) adjacent to it). This amounts to the contribution of the tuples with \(\sum_{1 \leq w \leq 2m}{shdeg(i_w)}=0\) being
\begin{equation}\tag{\(T2\)}\label{t2}
    \mathbb{E}[(\frac{1}{N}tr(R_1))^m]+O(m^2 \cdot (Im(z))^{-m} \cdot \frac{d^2}{n}).
\end{equation}
Thus, (\ref{t1}) alongside \(2^{1+14/6}<2^4=16,\) and (\ref{t2}) conclude the proof of (\ref{indbd}), contingent on justifying (\ref{fixedcontr}), a result undertaken next.

\subsection{Fixing Shared Degrees}\label{subsect4}

The goal of this subsection is proving (\ref{fixedcontr}).
\par
Begin by noting that
\begin{equation}\label{wantedz}
    0 \leq \mathbb{E}[Z_{1i_1}Z_{1i_2}...Z_{1i_{2m}}] \leq (\mathbb{E}[x^{2m}])^{\frac{1}{2}\sum_{1 \leq v \leq 2m}{shdeg(i_v)}}:
\end{equation}
symmetry renders that solely the preimages with
\begin{equation}\label{square}
    Z_{1i_1}Z_{1i_2}...Z_{1i_{2m}}=x_{j_1}^{2g'_1} x_{j_2}^{2g'_2} ... x_{j_l}^{2g'_l}
\end{equation}
for some \(g'_1,g'_2,\hspace{0.05cm}...\hspace{0.05cm},g'_l \in \mathbb{N},1 \leq j_1<j_2<\hspace{0.05cm}...\hspace{0.05cm}<j_l \leq n\) make a nonzero contribution. Independence and \(g'_s \leq m\) give
\[\mathbb{E}[Z_{1i_1}Z_{1i_2}...Z_{1i_{2m}}]=\mathbb{E}[x_{j_1}^{2g'_1}...x_{j_l}^{2g'_l}]=\mathbb{E}[x_{j_1}^{2g'_1}] \cdot \mathbb{E}[x_{j_2}^{2g'_2}] \cdot ... \cdot \mathbb{E}[x_{j_l}^{2g'_l}] \leq (\mathbb{E}[x^{2m}])^{t},\]
where \(t=|\{s:1 \leq s \leq l, g'_s \geq 2\}|.\) Finally, observe that
\[t \leq \frac{1}{2}\sum_{1 \leq v \leq 2m}{shdeg(i_v)},\] 
while Hölder's inequality gives \(1=\mathbb{E}[x_1^2] \leq (\mathbb{E}[x_1^{2m}])^{\frac{1}{2m}},\) together amounting to (\ref{wantedz}): for the former, there are at most \(\frac{1}{2}\sum_{1 \leq v \leq 2m}{shdeg(i_v)}\) indices \(j\) such that \(x_{j}\) appears in at least a pair \(Z_{1u},Z_{1v}\) with \(u,v\) not counterparts, and each index \(s\) with \(g'_s \geq 2\) falls within this category (\(g'_s=1\) when the aforesaid condition on \(j_s\) is violated). 
\par
Hence (\ref{fixedcontr}) will be complete once
\begin{equation}\label{wantedbound}
    N^{-m}\sum_{(*),1 \leq i_1,i_2,\hspace{0.05cm}...\hspace{0.05cm},i_{2m} \leq N}{\mathbb{E}[|R_{i_1i_2}R_{i_3i_4}...R_{i_{2m-1}i_{2m}}|]} \leq (Im(z))^{-m} \cdot (\frac{d^2 \cdot 2^{14m}}{n})^{\frac{1}{12}\sum_{1 \leq v \leq 2m}{shdeg(i_v)}}
\end{equation}
is shown to hold, where \((*)\) denotes that the summation is taken over tuples with \((shdeg(i_v))_{1 \leq v \leq 2m}\) fixed and \(\mathbb{E}[Z_{1i_1}Z_{1i_2}...Z_{1i_{2m}}] \ne 0.\) This is justified by (strong) induction on \(m\sum_{1 \leq v \leq 2m}{shdeg(i_v)} \in \mathbb{Z}_{\geq 0}.\)
\par
The base case \(\sum_{1 \leq v \leq 2m}{shdeg(i_v)}=0\) corresponds to \(i_{2l-1}=i_{2l}\) for all \(1 \leq l \leq m\) (otherwise \(\mathbb{E}[Z_{1i_1}Z_{1i_2}...Z_{1i_{2m}}]=\prod_{1 \leq l \leq m}{\mathbb{E}[Z_{1i_{2l-1}}Z_{1i_{2l}}]}=0\)), and the result is immediate:
\[N^{-m}\sum_{(*),1 \leq i_1,i_2,\hspace{0.05cm}...\hspace{0.05cm},i_{2m} \leq N}{\mathbb{E}[|R_{i_1i_2}R_{i_3i_4}...R_{i_{2m-1}i_{2m}}|]} \leq \mathbb{E}[(\frac{1}{N}\sum_{1 \leq j \leq N}{|R_{jj}|})^m] \leq (Im(z))^{-m}\]
insofar as for any Hermitian matrix \(M \in \mathbb{C}^{n \times n},\) \(|M_{jj}| \leq ||M||\) for all \(1 \leq j \leq n.\)
\par
Proceed with the induction step. If \(m=1,\) the result is clear because this entails \(i_1=i_2,\) whereby the above analysis applies. Suppose next \(k=\sum_{1 \leq v \leq 2m}{shdeg(i_v)}>0, m \geq 2.\) Take \(v \in \{1,2,\hspace{0.05cm}...\hspace{0.05cm},2m\}\) with
\[l=shdeg(i_v)=\min{\{shdeg(i_w): 1 \leq w \leq 2m, shdeg(i_w)>0\}},\] 
and consider the contribution of \(S(v)=\{v,v'\}\) to the sum in (\ref{wantedbound}). By discarding \(v,v'\) from the sets underlying the shared degrees, the product \(m\sum_{1 \leq v \leq 2m}{shdeg(i_v)}\) shrinks by at least \(\sum_{1 \leq v \leq 2m}{shdeg(i_v)}>~0,\) implying the induction hypothesis applies to the new tuples once their shared degrees are fixed. Since 
each shared degree decreases by at most \(2shdeg(v),\) there are at most
\[(2shdeg(v)+1)^{\frac{k}{shdeg(v)}} \leq 2^{2k}\] 
possibilities for the shared degrees of the remaining \(2m-2\) vertices (\(2^a \geq a+1\) for \(a \in \mathbb{N},\) and solely positive shared degrees can change, the number of such positions being at most \(\frac{k}{shdeg(v)}\) by the definition of \(v\)), and it suffices to show that once the new shared degrees are fixed, the contribution is at most
\[(Im(z))^{-m} \cdot (\frac{d^2 \cdot 2^{14m-24}}{n})^{\frac{k}{12}}.\]
For each tuple \((i_w)_{1 \leq w \leq 2m, w \not \in S(v)},\) there are at most \(\binom{n-l}{d-l} \cdot d^l \cdot (2m)^l\) possibilities for choosing \(i_v:\) given the rest, the vertex \(i_v\) shares \(l\) of the entries of the tuple \(b_{d,n}(i_v)\) with at most \(l\) of them, and so there are at most \((2m)^l \cdot d^l\) possibilities to choose these \(l\) indices with the rest of \(d-l\) positions being necessarily distinct from these \(l;\) lastly, notice that \(i_{v'}\) is fully determined by \((i_w)_{1 \leq w \leq 2m, w \ne v'}\) since \(Z_{1i_1}Z_{1i_2}...Z_{1i_{2m}}\) must be a square (i.e., (\ref{square}) holds) to ensure its expectation does not vanish. Because
\[\binom{n-l}{d-l} \cdot (2m)^l  \cdot d^l \leq N \cdot (\frac{8d^2m}{n})^l\]
from
\[\frac{\binom{n-l}{d-l}}{\binom{n}{d}} \leq \frac{\frac{n^{d-l}}{(d-l)!}}{\frac{(n-d)^d}{d!}}=\frac{n^{d-l} \cdot \frac{d!}{(d-l)!}}{(n-d)^d} \leq \frac{n^{d-l} \cdot d^{l}}{(n-d)^d} \leq (1+\frac{d}{n-d})^{d-l} \cdot (\frac{2d}{n})^l \leq e^{\frac{d^2}{n-d}} \cdot(\frac{2d}{n})^l\leq (\frac{4d}{n})^l\]
(using that \(d \leq \frac{n}{2}, \frac{d^2}{n-d} \leq \log{2}\) for \(n \geq n_3\)), it follows that for a given tuple \((i_w)_{1 \leq w \leq 2m, w \not \in S(v)},\)
\[\sum_{i_v,i_{v'}}{|R_{i_vi_{v'}}|} \leq \sqrt{N \cdot (\frac{8d^2m}{n})^l  \cdot \sum_{i_v}{|R_{i_vi_{v'}}|^2}} \leq \sqrt{N \cdot (\frac{8d^2m}{n})^l \cdot \frac{N}{(Im(z))^2}}=\frac{N}{Im(z)} \cdot (\frac{8d^2m}{n})^{\frac{l}{2}}\]
by employing that \(v'\) is fully determined by the rest of the tuple, and \(|R_{st}| \leq ||R|| \leq \frac{1}{Im(z)}\) for all \(1 \leq s,t \leq N.\) This and the induction hypothesis lead to an overall bound
\[\frac{1}{Im(z)} \cdot (\frac{8d^2m}{n})^{\frac{l}{2}} \cdot (Im(z))^{-(m-1)} \cdot (\frac{d^2 \cdot 2^{14m-24}}{n})^{\frac{1}{12}(k-6l)} \leq (Im(z))^{-m} \cdot (\frac{d^2 \cdot 2^{14m-24}}{n})^{\frac{k}{12}}\]
from \(8m \leq 2^{14m-24}\) as \(2^{14m-24} \geq 2^{2m}=8 \cdot 2^{2m-3} \geq 8 \cdot (2m-3+1) \geq 8m;\) \(\frac{d^2 \cdot 2^{14m-12}}{n} \leq 1\) for \(n \geq n_4\) (recall (\ref{small})), and 
\[\sum_{1 \leq w \leq 2m, w \not \in S(v)}{shdeg'(i_w)} \geq \sum_{1 \leq w \leq 2m}{shdeg(i_w)}-3shdeg(i_v)-3shdeg(i_v')=k-6l,\]
where \(shdeg'(i_w)\) denotes the shared degree of \(i_w\) upon excluding \(v,v'\) from the sets underlying the shared degrees: for \(u \not \in S(v),\) discarding \(i_v,i_{v'}\) decreases \(shdeg(i_u)\) exactly when there exists \(r \in \{1,2,\hspace{0.05cm}...\hspace{0.05cm},n\}\) that appears in the sets underlying \(b_{d,n}(i_u),b_{d,n}(i_v)\) or \(b_{d,n}(i_u),b_{d,n}(i_{v'})\) but not in any other such set except the one for the counterpart of \(u,\) i.e., in \(b_{d,n}(i_w), w \not \in S(u) \cup S(v):\) hence each \(r \in \{r_{vj},r_{v'j},1 \leq j \leq d\}\) can cause at most a drop of size \(2\) to the new sum of shared degrees. This completes the induction step.

\section{Beyond the Isotropic Case}\label{sectgen}

This section contains the proof of Theorem~\ref{mainth2}. This merges the main result in Bai and Zhou~\cite{baizhou} stated below and universality. 

\begin{theorem}\label{baizhouth}
    [Bai and Zhou~\cite{baizhou}] Suppose \(\tilde{Z} \in \mathbb{C}^{N \times p}\) has i.i.d. columns \((\tilde{Z}_j)_{1 \leq j \leq p} \subset \mathbb{C}^N\) with
    \par
    \((a)\) \(T=T(N)=\mathbb{E}[\tilde{Z}_1\tilde{Z}_1^*]\) having \(\sup_{N \in \mathbb{N}}{||T(N)||}<\infty\) and its empirical spectral distribution converging weakly to a deterministic probability distribution \(H,\) 
    \par
    \((b)\) \(\frac{1}{N^2}\mathbb{E}[|\tilde{Z}_1^{*}B\tilde{Z}_1-tr(TB)|^2]=o(1)\) for any deterministic symmetric matrix \(B \in \mathbb{R}^{N \times N},||B|| \leq 1,\) and
    \par
    \((c)\) \(\lim_{N \to \infty}{\frac{N}{p}}=\gamma \in (0,\infty).\) 
    \par
    Then the empirical spectral distribution of \(\tilde{S}=\frac{1}{p}\tilde{Z}^*\tilde{Z}\) converges weakly to a probability measure \(\mu\) almost surely, and the Stieltjes transform of \(\mu\) satisfies 
    \[m(z)=\int_{\mathbb{R}}{\frac{1}{z-t}d\mu(t)}=\int_{\mathbb{R}}{\frac{1}{z-t(1-\gamma+\gamma z m(z))}dH(t)} \hspace{1cm} (z \in \mathbb{C}^+).\]
\end{theorem}

\begin{remark}
The Stieltjes transform in this work differs by that in Bai and Zhou~\cite{baizhou} by sign, while in terms of dimensions, \((p,n)\) from the latter correspond to \((N,p)\) here (i.e., there are \(n\) samples of dimension \(p\) in \cite{baizhou}).
\end{remark}


Theorems~\ref{baizhouth} and \ref{mainth2} are alike in spirit, the key difference being the concentration conditions. In the former, this is encompassed by \((b),\) while in the latter, this is a byproduct of the random tensor product model and the restriction \(\frac{d}{n^{1/2}}=o(1).\) It will become more apparent below that these constraints are used to compute expectations that are vital towards showing the Stieltjes transform of the limit satisfies (\ref{steq}).
\par
In the remainder of this section, assume the conditions in Theorem~\ref{mainth2} hold (including (\ref{modelz})). 
Use anew the \(3\)-step strategy employed for Theorem~\ref{mainth}: due to the considerable overlap between the two proofs, solely the differences will be pointed out. Denote by \(\tilde{F}_n\) the empirical spectral distribution of \(\tilde{S},\) let 
\begin{equation}\label{bdtop}
    \mathcal{K}=\sup_{n \in \mathbb{N}}{||T(n)||}<\infty,
\end{equation}
and for \(z \in \mathbb{C}^+\) fixed, take \(\tilde{R}_j=(zI-\tilde{S}+\frac{1}{p}\tilde{Z}_j\tilde{Z}^T_j)^{-1}\) when \(1 \leq j \leq p.\)
\par
\(1.\) The first step, encompassed by (\ref{as}), continues to hold since as discussed in its proof, independence of columns suffice:
\[m_{\tilde{F}_n}(z)-\mathbb{E}[m_{\tilde{F}_n}(z)] \xrightarrow[]{a.s.} 0, \hspace{0.5cm} \forall z \in \mathbb{C}^{+}.\]
\par
\(2.\) The second step, encapsulated by Proposition~\ref{propstep2}, can be modified by adapting the proof of Theorem~\ref{baizhouth}. The thrust of the former result is (\ref{concl2}), which is exploited in the third step to show the desired convergence. In the current case, the limiting probability measure \(\mu\) does not have as neat a description as the Marchenko-Pastur laws do,
the source of the additional complexity being the covariance matrix \(T\) that is no longer the identity matrix. This difficulty can be handled with an ingenious use of the idea in Proposition~\ref{propstep2}, 
\[-1+z\mathbb{E}[m_{F_n}(z)]=-\frac{1}{\gamma_n}+\frac{1}{\gamma_n}\mathbb{E}[\frac{1}{1-\frac{1}{p}Z_1^TR_1Z_1}],\]
derived from
\[-(zI-S)+zI=\frac{1}{p}\sum_{1 \leq j \leq p}{Z_jZ_j^T}\]
(see Lemma~\ref{ledoitpechelemma} for proof). Concretely, the lack of isotropy is bypassed by two uses of such identities in Bai and Zhou~\cite{baizhou}: although the original argument appears to be flawed, it can be corrected as explained in subsection~\ref{4.1} below, while subsection~\ref{4.2} demonstrates how the proof of Proposition~\ref{propstep2} can be modified to justify its following analogue.

\begin{proposition}\label{propstep2e}
     Suppose the assumptions in Theorem~\ref{mainth2} hold. 
     Then 
     \begin{equation}
         \mathbb{E}[m_{\tilde{F}_n}(z)]=m(z)+o(1),
     \end{equation}
     when \(Im(z) \geq 1+4\mathcal{K}\gamma\) and where \(m\) is defined by (\ref{steq}).
\end{proposition}

\par
\(3.\) The third step remains valid verbatim: the uniqueness of \(\mu\) (justified in Bai and Zhou~\cite{baizhou}: see step \(3\) is section \(3\) therein) and of analytical continuations yields all weak limits are the one described in Theorem~\ref{mainth2}, from which the desired conclusion ensues.
\par
In the rest of this section, subsection~\ref{4.1} corrects the argument for Theorem~\ref{baizhouth} from Bai and Zhou~\cite{baizhou}, and subsection~\ref{4.2} consists of the justification of Proposition~\ref{propstep2e}.

\subsection{A Crucial Identity}\label{4.1}

The issue in Bai and Zhou~\cite{baizhou} in the proof of their main result is related to the definition of the matrix \(K\) in (their) step \(2.\) The authors let 
\[K=T(1-\frac{1}{p}tr(\tilde{R}_jT))^{-1},\] 
which a priori is not well-defined as \(j\) appears to be a generic column index, in this case, an element of \(\{1,2,\hspace{0.05cm}...\hspace{0.05cm},p\},\) and treat this matrix as if it were independent of \(Z_j\) for all \(j,\) 
a property needed at the hour of justifying certain differences remain negligible (i.e., \(d_{j,2},d_{j,3}:\) see also (\ref{diffdef}) in the next subsection). 
\par
To correct this, let 
\[K(j)=T(1-\frac{1}{p}tr(\tilde{R}_jT))^{-1} \hspace{0.2cm} (1 \leq j \leq p), \hspace{0.5cm} \tilde{K}=K(1):\]
notice that \((K(j))_{1 \leq j \leq p}\) are identically distributed with \(\tilde{Z}_j\) and \(K(j)\) independent for all \(1 \leq j \leq p.\) In what follows, solely the real case (\((Z_j)_{1 \leq j \leq p} \subset \mathbb{R}^N\)) is considered insofar as this work focuses exclusively on this situation: however, the argument can be easily adapted to the complex case (\((Z_j)_{1 \leq j \leq p} \subset \mathbb{C}^N\)). 
\par
Recall the reasoning in Bai and Zhou~\cite{baizhou} that leads to the main identity in the section treating step \(2,\) equation \(3.3\) therein ((\ref{3.3corr}) below). Since
\[(zI-\tilde{S})-(zI-\tilde{K})=\tilde{K}-\frac{1}{p}\sum_{1 \leq j \leq p}{\tilde{Z}_j\tilde{Z}_j^T},\]
the resolvent identity \(A^{-1}-B^{-1}=A^{-1}(B-A)B^{-1}\) gives
\[(zI-\tilde{K})^{-1}-(zI-\tilde{S})^{-1}=(zI-\tilde{K})^{-1}\tilde{K}(zI-\tilde{S})^{-1}-\frac{1}{p}\sum_{1 \leq j \leq p}{(zI-\tilde{K})^{-1}\tilde{Z}_j\tilde{Z}_j^T(zI-\tilde{S})^{-1}},\]
which yields together with 
\begin{equation}\label{multid}
    \tilde{Z}_j^T(zI-\tilde{S})^{-1}=\frac{\tilde{Z}_j^T\tilde{R}_j}{1-\frac{1}{p}\tilde{Z}_j^T\tilde{R}_jZ_j},
\end{equation}
the following crucial identity
\begin{equation}\tag{\(BZ\)}\label{3.3corr}
    (zI-\tilde{K})^{-1}-(zI-\tilde{S})^{-1}=(zI-\tilde{K})^{-1}\tilde{K}(zI-\tilde{S})^{-1}-\frac{1}{p}\sum_{1 \leq j \leq p}{\frac{(zI-\tilde{K})^{-1}\tilde{Z}_j\tilde{Z}_j^T\tilde{R}_j}{1-\frac{1}{p}\tilde{Z}_j^T\tilde{R}_j\tilde{Z}_j}}:
\end{equation}
here (\ref{multid}) is a consequence of
\[\tilde{R}_j-(zI-\tilde{S})^{-1}=-\tilde{R}_j\frac{1}{p}\tilde{Z}_j\tilde{Z}^T_j(zI-\tilde{S})^{-1},\]
from which
\[\tilde{Z}_j^T\tilde{R}_j-\tilde{Z}_j^T(zI-\tilde{S})^{-1}=-\frac{1}{p}\tilde{Z}_j^T\tilde{R}_j\tilde{Z}_j \cdot \tilde{Z}^T_j(zI-\tilde{S})^{-1},\]
and note that \(1-\frac{1}{p}\tilde{Z}_j^T\tilde{R}_j\tilde{Z}_j=0\) cannot occur as this would entail \(\tilde{Z}_j^T\tilde{R}_j=0,\) whereby 
\[\tilde{Z}_j=0, \hspace{0.5cm} 0=1-\frac{1}{p}\tilde{Z}_j^T\tilde{R}_j\tilde{Z}_j=1-0=1.\] 
\par
Multiplying (\ref{3.3corr}) on the left by \(T^l\) for \(l \in \{0,1\},\) taking traces of both sides and dividing by \(N\) leads to
\begin{equation}\tag{\(BZ_l\)}\label{bzl}
    \frac{1}{N}tr(T^l(zI-\tilde{K})^{-1})-\frac{1}{N}tr(T^l(zI-\tilde{S})^{-1})=\frac{1}{N}tr(T^l(zI-\tilde{K})^{-1}\tilde{K}(zI-\tilde{S})^{-1})-\frac{1}{N}\sum_{1 \leq j \leq p}{\frac{\frac{1}{p}\tilde{Z}_j^T\tilde{R}_jT^l(zI-\tilde{K})^{-1}\tilde{Z}_j}{1-\frac{1}{p}\tilde{Z}_j^T\tilde{R}_jZ_j}}.
\end{equation}
Bai and Zhou~\cite{baizhou} argue next that the expectation of the right-hand side tends to \(0\) as \(N \to \infty:\) however, this relies on \(\tilde{R}_jT^l(zI-\tilde{K})^{-1}\) being independent of \(\tilde{Z}_j,\) which is not guaranteed for all column indices \(j.\) 
This convergence subsequently yields
\[\mathbb{E}[\frac{1}{N}tr(T^l(zI-\tilde{K})^{-1})]-\mathbb{E}[\frac{1}{N}tr(T^l(zI-\tilde{S})^{-1})]=o(1) \hspace{1cm} (l \in \{0,1\}),\]
from which the desired equation (\ref{steq}) is derived. The problem regarding the expectation of the right-hand side term of (\ref{bzl}), rooted in the definition of \(K,\) can be resolved with the aid of concentration: since \((K(j))_{1 \leq j \leq p}\) are identically distributed, it suffices to show
\begin{equation}\label{fixbz}
    \mathbb{E}[\frac{\frac{1}{N}\tilde{Z}_j^T\tilde{R}_jT^l(zI-\tilde{K})^{-1}\tilde{Z}_j}{1-\frac{1}{p}\tilde{Z}_j^T\tilde{R}_jZ_j}-\frac{\frac{1}{N}\tilde{Z}_j^T\tilde{R}_jT^l(zI-\tilde{K}_0)^{-1}\tilde{Z}_j}{1-\frac{1}{p}\tilde{Z}_j^T\tilde{R}_jZ_j}]=o(1)
\end{equation}
for \(\tilde{K}_0=T(1-\frac{1}{p}\mathbb{E}[tr(\tilde{R}_jT)])^{-1}\) (notice that this matrix is independent of \(j\)), insofar as this will entail that after taking expectations of both sides in (\ref{bzl}), \(\tilde{K}\) can be replaced by \(K(j)\) in the \(j^{th}\) summand on the right 
due to the \(j^{th}\) term being
\[\mathbb{E}[\frac{\frac{1}{N}\tilde{Z}_1^T\tilde{R}_1T^l(zI-\tilde{K}_0)^{-1}\tilde{Z}_1}{1-\frac{1}{p}\tilde{Z}_1^T\tilde{R}_1Z_1}]+o(1).\]
\vspace{0.3cm}
\par
Identity (\ref{fixbz}) is the goal for the remainder of this subsection. As stated in Bai and Zhou~\cite{baizhou},
\begin{equation}\label{prodinv}
    \frac{1}{|1-\frac{1}{p}\tilde{Z}_j^T\tilde{R}_j\tilde{Z}_j|} \leq \frac{|z|}{Im(z)},
\end{equation}
whereby showing 
\[\mathbb{E}[\frac{1}{N}|\tilde{Z}_j^T\tilde{R}_jT^l((zI-\tilde{K})^{-1}-(zI-\tilde{K}_0)^{-1})\tilde{Z}_j|]=o(1)\]
is enough for concluding (\ref{fixbz}): (\ref{prodinv}) is a consequence of the inequality below bound \(3.4\) in \cite{baizhou}, which entails 
\begin{equation}\label{consbaiz}
    Im(-z \cdot v^*\tilde{R}_jv) \geq 0
\end{equation}
for all random vectors \(v \in \mathbb{C}^n\) independent of \(\tilde{R}_j\) insomuch as (\ref{consbaiz}) then gives
\[Im(z) \leq Im(z(1-v^*\tilde{R}_jv)) \leq |z| \cdot |1-v^*\tilde{R}_jv|.\]
\par
Start with
\[\frac{1}{N}|\tilde{Z}_j^T\tilde{R}_jT^l((zI-\tilde{K})^{-1}-(zI-\tilde{K}_0)^{-1})\tilde{Z}_j| \leq\]
\begin{equation}\label{id1}
    \leq \frac{1}{N}||\tilde{Z}_j||^2 \cdot \frac{1}{Im(z) \cdot c^2(\mathcal{K},z)} \cdot ||T||^{l+1} \cdot |(1-\frac{1}{p}tr(\tilde{R}_jT))^{-1}-(1-\frac{1}{p}\mathbb{E}[tr(\tilde{R}_jT)])^{-1}|,
\end{equation}
where 
\[c(\alpha,z)=\min_{(r,\lambda):Im(z/r) \geq Im(z),\lambda \in [0,\alpha]}{|z-r\lambda|} \in (0,\infty) \hspace{0.5cm} (\alpha>0):\]
note that continuity and compactness (\(Im(z/r) \leq \frac{|z|}{|r|}\) entails \(|r| \leq \frac{|z|}{Im(z)},\) and the absolute value of interest can be seen as a function defined on \(\mathbb{R}^3\)) give the minimum exists and is attained at some \((r_0,\lambda_0)\) and \(|z-r_0\lambda_0|>0\) (else, \(z=r_0\lambda_0,\) whereby \(r_0 \ne 0\) from \(z \ne 0,\) and so \(Im(z/r_0)=Im(\lambda_0)=0,\) contradicting \(Im(z/r_0) \geq Im(z)\)).
To justify (\ref{id1}), notice the left-hand side term is at most
\[\frac{1}{N}||\tilde{Z}_j||^2 \cdot ||\tilde{R}_jT^l((zI-\tilde{K})^{-1}-(zI-\tilde{K}_0)^{-1})|| \leq \frac{1}{N}||\tilde{Z}_j||^2 \cdot ||\tilde{R}_j|| \cdot ||T||^l \cdot ||(zI-\tilde{K})^{-1}-(zI-\tilde{K}_0)^{-1}||\]
and it is enough to show
\[||(zI-\tilde{K})^{-1}-(zI-\tilde{K}_0)^{-1}|| \leq \frac{||T||}{c^2(\mathcal{K},z)} \cdot |(1-\frac{1}{p}tr(\tilde{R}_1T))^{-1}-(1-\frac{1}{p}\mathbb{E}[tr(\tilde{R}_1T)])^{-1}|.\]
This last claim can be justified as follows: for \(a,b \in \mathbb{C},\)
\[(zI-aT)^{-1}-(zI-bT)^{-1}=U^Tdiag((\frac{\lambda_s(a-b)}{(z-a\lambda_s)(z-b\lambda_s)})_{1 \leq s \leq N})U\] 
for a spectral decomposition \(T=U^Tdiag((\lambda_s)_{1 \leq s \leq N})U,\) and
\[|\frac{\lambda_s(a-b)}{(z-a\lambda_s)(z-b\lambda_s)}| \leq ||T|| \cdot |a-b| \cdot \frac{1}{\min_{1 \leq r \leq N}{|z-a\lambda_r|} \cdot \min_{1 \leq r \leq N}{|z-b\lambda_r|}} \leq \frac{||T|| \cdot |a-b|}{c^2(\mathcal{K},z)}\]
when \(a=(1-\frac{1}{p}tr(\tilde{R}_jT))^{-1},b=(1-\frac{1}{p}\mathbb{E}[tr(\tilde{R}_jT)])^{-1}\) by using that \(Im(z/a)=Im(z(1-\frac{1}{p}tr(\tilde{R}_jT)))\geq Im(z)\) (the bound ensues from \(T\) being positive semidefinite and (\ref{consbaiz})) with linearity of expectation entailing \(Im(z/b) \geq Im(z).\) 
\par
Next, as mentioned above, 
\begin{equation}\label{anorm}
    \frac{1}{|1-\frac{1}{p}tr(\tilde{R}_1T)|} \leq \frac{|z|}{Im(z)},
\end{equation}
whereby (\ref{id1}) can be changed to
\[\frac{1}{N}|\tilde{Z}_j^T\tilde{R}_jT^l((zI-\tilde{K})^{-1}-(zI-\tilde{K}_0)^{-1})\tilde{Z}_j| \leq \frac{||T||^{l+1} \cdot |z|^2}{(Im(z))^3 \cdot c^2(\mathcal{K},z)} \cdot \frac{1}{N}||\tilde{Z}_j||^2 \cdot |\frac{1}{p}tr(\tilde{R}_1T)-\frac{1}{p}\mathbb{E}[tr(\tilde{R}_1T)]|.\]
To conclude (\ref{fixbz}), by virtue of (\ref{bdtop}), it suffices to show that for \(1 \leq j \leq p,\)
\begin{equation}\label{last}
    \mathbb{E}[\frac{1}{N}||\tilde{Z}_j||^2 \cdot |\frac{1}{p}tr(\tilde{R}_1T)-\frac{1}{p}\mathbb{E}[tr(\tilde{R}_1T)]|]=o(1).
\end{equation}
\par
If \(j=1,\) then (\ref{last}) follows from independence and Hölder's inequality,
\[\mathbb{E}[\frac{1}{N}||\tilde{Z}_1||^2 \cdot |\frac{1}{p}tr(\tilde{R}_1T)-\frac{1}{p}\mathbb{E}[tr(\tilde{R}_1T)]|]=\mathbb{E}[\frac{1}{N}||\tilde{Z}_1||^2] \cdot \mathbb{E}[|\frac{1}{p}tr(\tilde{R}_1T)-\frac{1}{p}\mathbb{E}[tr(\tilde{R}_1T)]|] \leq\]
\[\leq ||T|| \cdot (\frac{16||T||^{4}}{(Im(z))^4 \cdot p^{2}})^{1/4}=O(p^{-1/2})=o(1)\]
since an analogous rationale to the one in step \(1\) (i.e., the proof of (\ref{as})) 
alongside lemma \(2.6\)\footnote{For \(v \in \mathbb{C}^N,\) \(A,B \in \mathbb{C}^{N \times N}.\)} in Bai and Silverstein~\cite{silv2},
\begin{equation}\label{2.6}
    |tr((zI-B)^{-1}A-(zI-B+vv^*)^{-1}A)| \leq \frac{||A||}{Im(z)}
\end{equation}
give
\[\mathbb{E}[|\frac{1}{p}tr(\tilde{R}_jT)-\frac{1}{p}\mathbb{E}[tr(\tilde{R}_jT)]|^{4}] \leq \frac{16p^2 \cdot ||T||^{4}}{(Im(z))^4 \cdot p^{4}}.\]

\par
Else \(j \ne 1,\) and define \(\tilde{R}_{j1}=(zI-\tilde{S}+\frac{1}{p}\tilde{Z}_1\tilde{Z}_1^T+\frac{1}{p}\tilde{Z}_j\tilde{Z}_j^T)^{-1}:\) then
\[\mathbb{E}[\frac{1}{N}||\tilde{Z}_j||^2 \cdot |\frac{1}{p}tr(\tilde{R}_1T)-\frac{1}{p}\mathbb{E}[tr(\tilde{R}_1T)]|] \leq \mathbb{E}[\frac{1}{N}||\tilde{Z}_j||^2 \cdot |\frac{1}{p}tr(\tilde{R}_{j1}T)-\frac{1}{p}\mathbb{E}[tr(\tilde{R}_{j1}T)]|]+\]
\[+\mathbb{E}[\frac{1}{N}||\tilde{Z}_j||^2 \cdot |\frac{1}{p}tr(\tilde{R}_{1}T)-\frac{1}{p}tr(\tilde{R}_{j1}T)|]+\mathbb{E}[\frac{1}{N}||\tilde{Z}_j||^2 \cdot |\frac{1}{p}\mathbb{E}[tr(\tilde{R}_{1}T)]-\frac{1}{p}\mathbb{E}[tr(\tilde{R}_{j1}T)]|]:=I_e+II_e+III_e,\]
from which (\ref{last}) ensues insofar as by a similar rationale as to that for \(j=1,\) \(I_e=O(p^{-1/2}),\) and two applications of (\ref{2.6}) give
\[II_e+III_e \leq 2 \cdot \frac{||T||}{Im(z)} \cdot \mathbb{E}[\frac{1}{Np}||\tilde{Z}_j||^2] \leq  \frac{2||T||^2}{Im(z)} \cdot \frac{1}{p}.\]

\subsection{Swapping Random Variables}\label{4.2}

Return now to the model in Theorem~\ref{mainth2}. Condition \((b)\) in Theorem~\ref{baizhouth} is used when showing that
\[\mathbb{E}[|\frac{1}{p}\tilde{Z}_j^T\tilde{R}_jT^l(zI-\tilde{K}(j))^{-1}\tilde{Z}_j-(1-\frac{1}{p}\tilde{Z}_j^T\tilde{R}_jZ_j)\cdot \frac{1}{p}tr(T^l(zI-\tilde{K}(j))^{-1}\tilde{K}(j)(zI-\tilde{S})^{-1})|]=o(1)\]
for \(1 \leq j \leq p, l \in \{0,1\},\) from which the conclusion in Proposition~\ref{propstep2e} is deduced for \(z \in \mathbb{C}^{+}.\) Namely,
\begin{equation}\label{diffdef}
    \frac{1}{p}\tilde{Z}_j^T\tilde{R}_jT^l(zI-\tilde{K}(j))^{-1}\tilde{Z}_j-(1-\frac{1}{p}\tilde{Z}_j^T\tilde{R}_jZ_j)\cdot \frac{1}{p}tr(T^l(zI-\tilde{K}(j))^{-1}\tilde{K}(j)(zI-\tilde{S})^{-1}):=d_{j,1}+d_{j,2}+d_{j,3}
\end{equation}
for
\[d_{j,1}=\frac{1}{p}tr(\tilde{R}_j(zI-\tilde{K}(j))^{-1}T^{l+1})-\frac{1}{p}tr((zI-\tilde{S})^{-1}(zI-\tilde{K}(j))^{-1}T^{l+1}),\]
\[d_{j,2}=\frac{1}{p}\tilde{Z}_j^T\tilde{R}_jT^l(zI-\tilde{K}(j))^{-1}\tilde{Z}_j-\frac{1}{p}tr(\tilde{R}_j(zI-\tilde{K}(j))^{-1}T^{l+1}),\]
\[d_{j,3}=\frac{1}{p}tr((zI-\tilde{S})^{-1}(zI-\tilde{K}(j))^{-1}T^{l+1}) \cdot (1-\frac{1-\frac{1}{p}\tilde{Z}_j^T\tilde{R}_jZ_j}{1-\frac{1}{p}tr(\tilde{R}_jT)})\]
by employing that \(\tilde{K}(j)\) and \(T\) commute with condition \((b)\) applied to show \(\mathbb{E}[|d_{j,2}|]=o(1),\mathbb{E}[|d_{j,3}|]=o(1).\) 
\par
In the current case, condition \((a)\) in Theorem~\ref{baizhouth} remains valid (and condition \((c)\) is immediate) since
\[\mathbb{E}[\tilde{Z}_1\tilde{Z}_1^T]=\mathbb{E}[T^{1/2}Z_0Z_0^TT^{1/2}]=T^{1/2}\mathbb{E}[Z_0Z_0^T]T^{1/2}=T^{1/2}T^{1/2}=T,\]
and proving that under the conditions in Theorem~\ref{mainth2} and \(Im(z) \geq 1+4\mathcal{K}\gamma,\) 
\begin{equation}\tag{\(U\)}\label{unmomeq}
    \mathbb{E}[\frac{\frac{1}{p}\tilde{Z}_1^T\tilde{R}_1T^l(zI-\tilde{K})^{-1}\tilde{Z}_1}{1-\frac{1}{p}\tilde{Z}_1^T\tilde{R}_1\tilde{Z}_1}]=\mathbb{E}[\frac{\frac{1}{p}y^T\tilde{R}_1T^l(zI-\tilde{K})^{-1}y}{1-\frac{1}{p}y^T\tilde{R}_1y}]+o(1)
\end{equation}
for \(y \in \mathbb{R}^N,y=T^{1/2}y_0, y_0 \overset{d}{=}N(0,I),\) \(y_0\) independent of \((Z_j)_{1 \leq j \leq p},\) is enough insofar as condition \((b)\) holds for \(y:\)
\[\frac{1}{N^2}\mathbb{E}[|y^{T}By-tr(TB)|^2]=\frac{1}{N^2}\mathbb{E}[|y_0^{T}T^{1/2}BT^{1/2}y_0-tr(T^{1/2}BT^{1/2})|^2]=\]
\[=\frac{1}{N^2}\sum_{1 \leq j \leq N}{(3-2\cdot 1+1)\cdot (T^{1/2}BT^{1/2})^2_{jj}} \leq \frac{2N \cdot ||T^{1/2}BT^{1/2}||^2}{N^2} \leq \frac{2||T||^2}{N}\]
for any deterministic symmetric matrix \(B \in \mathbb{R}^{N \times N}\) with \(||B|| \leq 1.\) Thus, the rest of the argument in Theorem~\ref{baizhouth} can be used to complete Proposition~\ref{propstep2e}: the remainder of this subsection derives (\ref{unmomeq}).
\vspace{0.5cm}
\par
Identity (\ref{unmomeq}) can be justified in the same vein as Proposition~\ref{propstep2} was. Suppose \(Im(z) \geq 1+4\mathcal{K}\gamma,\) and use the same truncation, i.e.,
\[1=\chi_{||Z_j||^2>2N}+\chi_{||Z_j||^2 \leq 2N}.\]
\par
\((i)\) Consider first the tail and reason as in Lemma~\ref{tailev}: identity (\ref{eq1234}) holds in this case as well, giving
\[\frac{1}{1-\frac{1}{p}\tilde{Z}_j^T\tilde{R}_j\tilde{Z}_j}=1+\frac{1}{p}\tilde{Z}_j^T(zI-\tilde{S})^{-1}\tilde{Z}_j,\]
from which
\[|\frac{\frac{1}{p}\tilde{Z}_1^T\tilde{R}_1T^l(zI-\tilde{K})^{-1}\tilde{Z}_1}{1-\frac{1}{p}\tilde{Z}_1^T\tilde{R}_1\tilde{Z}_1}| \leq (1+\frac{1}{p}||\tilde{Z}_1||^2 \cdot \frac{1}{Im(z)}) \cdot \frac{1}{p}||\tilde{Z}_1||^2 \cdot ||\tilde{R}_1T^l(zI-\tilde{K})^{-1}|| \leq\]
\[\leq (1+\frac{||T||}{p} \cdot ||Z_1||^2 \cdot \frac{1}{Im(z)}) \cdot \frac{||T||}{p} \cdot ||Z_1||^2 \cdot \frac{1}{Im(z)} \cdot ||T||^l \cdot \frac{Im(z)+||T||}{(Im(z))^2}\]
by using bound \(3.4\) in Bai and Zhou~\cite{baizhou}, 
\[||(zI-\tilde{K})^{-1}|| \leq \frac{Im(z)+||T||}{(Im(z))^2}.\] 
Markov's inequality then yields
\[|\mathbb{E}[\frac{\frac{1}{p}\tilde{Z}_1^T\tilde{R}_1T^l(zI-\tilde{K})^{-1}\tilde{Z}_1}{1-\frac{1}{p}\tilde{Z}_1^T\tilde{R}_1\tilde{Z}_1} \cdot \chi_{||Z_1||^2>2N}]| \leq \frac{||T||^{l+1}\cdot (||T||+Im(z))}{(Im(z))^3} \cdot \mathbb{E}[\frac{1}{p}||Z_1||^2 \cdot \chi_{||Z_1||^2>2N}]+\]
\[+\frac{||T||^{l+2} \cdot (||T||+Im(z))}{(Im(z))^4} \cdot \mathbb{E}[(\frac{1}{p}||Z_1||^2)^2 \cdot \chi_{||Z_1||^2>2N}],\]
and
\[\mathbb{E}[\frac{1}{N}||Z_1||^2 \cdot \chi_{||Z_1||^2>2N}] \leq \frac{1}{2}\mathbb{E}[(\frac{1}{N}||Z_1||^2)^2 \cdot \chi_{||Z_1||^2>2N}],\]
\[\mathbb{E}[(\frac{1}{N}||Z_1||^2)^2  \cdot\chi_{||Z_1||^2>2N}]=\mathbb{P}(||Z_1||^2>2N)+\mathbb{E}[((\frac{1}{N}||Z_1||^2)^2-1) \cdot \chi_{||Z_1||^2>2N}] \leq\]
\[\leq Var(\frac{1}{N}||Z_1||^2)+2\mathbb{E}[\frac{1}{N}||Z_1||^2 \cdot |\frac{1}{N}||Z_1||^2-1|] \leq Var(\frac{1}{N}||Z_1||^2)+2\sqrt{\mathbb{E}[(\frac{1}{N}||Z_1||^2)^2 \cdot (\frac{1}{N}||Z_1||^2-1)^2]}\]
via \(0 \leq x^2-1=(x+1)(x-1) \leq 2x(x-1)\) for \(x \geq 1,\) and Cauchy-Schwarz inequality. These two bounds and (\ref{bdtop}) entail
\begin{equation}\tag{\(E1e\)}\label{ext1}
    \mathbb{E}[\frac{\frac{1}{p}\tilde{Z}_1^T\tilde{R}_1T^l(zI-\tilde{K})^{-1}\tilde{Z}_1}{1-\frac{1}{p}\tilde{Z}_1^T\tilde{R}_1\tilde{Z}_1} \cdot \chi_{||Z_1||^2>2N}]=o(1)
\end{equation}
insomuch as Lemmas~\ref{lenlemma},~\ref{lemmamom} imply \(Var(\frac{1}{N}||Z_1||^2)=o(1),\mathbb{E}[(\frac{1}{N}||Z_1||^2)^2 \cdot (\frac{1}{N}||Z_1||^2-1)^2]=o(1)\) respectively, the latter following from
\[\mathbb{E}[(\frac{1}{N}||Z_1||^2)^2 \cdot (\frac{1}{N}||Z_1||^2-1)^2]=\mathbb{E}[(\frac{1}{N}||Z_1||^2)^4]-2\mathbb{E}[(\frac{1}{N}||Z_1||^2)^3]+\mathbb{E}[(\frac{1}{N}||Z_1||^2)^2]=\]
\[=1+o(1)-2(1+o(1))+1+o(1)=o(1).\]
\par
\((ii)\) Continue now with the remainder of the expectation of interest,
\[\mathbb{E}[\frac{\frac{1}{p}\tilde{Z}_1^T\tilde{R}_1T^l(zI-\tilde{K})^{-1}\tilde{Z}_1}{1-\frac{1}{p}\tilde{Z}_1^T\tilde{R}_1\tilde{Z}_1} \cdot \chi_{||Z_1||^2 \leq 2N}].\]
Similarly to the proof of Lemma~\ref{lemmasmall}, the inequality
\[\frac{1}{p}|\tilde{Z}_1^T\tilde{R}_1\tilde{Z}_1| \leq \frac{1}{p}||Z_1||^2 \cdot ||T|| \cdot ||\tilde{R}_1|| \leq \frac{1}{p}||Z_1||^2 \cdot \frac{\mathcal{K}}{Im(z)}\]
entails that when \(Im(z) \geq 4\mathcal{K}\gamma_n,\) 
\[\frac{1}{p}|\tilde{Z}_1^T\tilde{R}_1\tilde{Z}_1| \cdot \chi_{||Z_1||^2 \leq 2N} \leq \frac{1}{p} \cdot 2N \cdot \frac{\mathcal{K}}{Im(z)} \leq \frac{1}{2}.\]
The aforesaid condition on \(z\) holds for \(n \geq n_0,\) whereby the rationale in subsection~\ref{subsect2}
can be used almost verbatim. Namely,
\begin{equation}\tag{\(E2ae\)}\label{bulk1e}
    \mathbb{E}[\frac{\frac{1}{p}\tilde{Z}_1^T\tilde{R}_1T^l(zI-\tilde{K})^{-1}\tilde{Z}_1}{1-\frac{1}{p}\tilde{Z}_1^T\tilde{R}_1\tilde{Z}_1} \cdot \chi_{||Z_1||^2 \leq 2N}]=\Sigma_{1e}+\Sigma_{2e}-\Sigma_{3e},
\end{equation}
\[\Sigma_{1e}=\sum_{0 \leq m \leq M}{\mathbb{E}[Y(\frac{1}{p}\tilde{Z}_1^T\tilde{R}_1\tilde{Z}_1)^m]}, \hspace{0.5cm} \Sigma_{2e}=\mathbb{E}[\frac{Y(\frac{1}{p}\tilde{Z}_1^T\tilde{R}_1\tilde{Z}_1)^{M+1}}{1-\frac{1}{p}\tilde{Z}_1^T\tilde{R}_1\tilde{Z}_1} \cdot \chi_{||Z_1||^2 \leq 2N}], \hspace{0.2cm}\]
\[\Sigma_{3e}=\sum_{0 \leq m \leq M}{\mathbb{E}[Y(\frac{1}{p}\tilde{Z}_1^T\tilde{R}_1\tilde{Z}_1)^m\cdot \chi_{||Z_1||^2>2N}]}\]
for \(Y=\frac{1}{p}\tilde{Z}_1^T\tilde{R}_1T^l(zI-\tilde{K})^{-1}\tilde{Z}_1.\) 
By an argument similar to the one employed for (\ref{bulk2}), 
\begin{equation}\tag{\(E2be\)}\label{bulk2e}
    \Sigma_{2e}=o(1), \hspace{0.5cm} \Sigma_{3e}=o(1):
\end{equation}
the argument for \(\Sigma_{2e}\) is identical, while for \(\Sigma_{3e},\) 
\(\gamma_n \cdot |\tilde{Z}_1^T\tilde{R}_1\tilde{Z}_1| \leq ||Z_1||^2\) suffices (due to the sum in the bound (\ref{sigma3bd})). The latter is guaranteed by \(\gamma_n \cdot ||\tilde{R}_1|| \cdot ||T|| \leq 1,\) which holds for \(n\) sufficiently large from
\[\gamma_n \cdot ||\tilde{R}_1|| \cdot ||T|| \leq \gamma_n \cdot \frac{1}{Im(z)} \cdot \mathcal{K} \leq \frac{\gamma_n \cdot \mathcal{K}}{1+4\gamma \cdot \mathcal{K}}.\]
\par
The remaining component, \(\Sigma_{1e},\) is universal in the following sense:
\begin{equation}\label{rel1}
    \Sigma_{1e}=\sum_{0 \leq m \leq M}{\mathbb{E}[Y(\frac{1}{p}\tilde{Z}_1^T\tilde{R}_1\tilde{Z}_1)^m]}=\sum_{0 \leq m \leq M}{\mathbb{E}[Y_G(\frac{1}{p}y^T\tilde{R}_1y)^m]}+o(1)
\end{equation}
for \(Y_G=\frac{1}{p}y^T\tilde{R}_1T^l(zI-\tilde{K})^{-1}y.\) This is a consequence of the moment computations behind (\ref{indbd}), i.e.,
\[\mathbb{E}[(\frac{1}{N}Z_1^TR_1Z_1)^m]=\mathbb{E}[(\frac{1}{N}tr(R_1))^m]+(Im(z))^{-m} \cdot [O((16Cm)^{m} \cdot (\frac{d}{n^{1/2}})^{1/3})+O(m^2 \cdot \frac{d^2}{n})].\]
Identity (\ref{sumexp}) is a special case of the following result,
\begin{equation}\label{genexp}
    \mathbb{E}[\prod_{1 \leq j \leq q}{Z_1^TW_jZ_1}]=\sum_{1 \leq i_1,i_2,\hspace{0.05cm}...\hspace{0.05cm},i_{2q} \leq N}{\mathbb{E}[(W_1)_{i_1i_2}(W_2)_{i_3i_4}...(W_q)_{i_{2m-1}i_{2m}}] \cdot \mathbb{E}[Z_{1i_1}Z_{1i_2}...Z_{1i_{2q}}]}
\end{equation}
for any random matrices \(W_1,W_2,\hspace{0.05cm}...\hspace{0.05cm},W_q \in \mathbb{R}^{n \times n}\) that are independent of \(Z_1\) (by conditioning on \(\sigma(W_{1},W_2,\hspace{0.05cm}...\hspace{0.05cm},W_q),\) a \(\sigma\)-algebra independent of \(\sigma(Z_1),\) and the tower property): namely, (\ref{sumexp}) is (\ref{genexp}) for 
\[W_1=W_2=...=W_m=\frac{1}{p}R_1.\] 
In the current case, the term corresponding to \(m\) in the left-hand side of (\ref{rel1}) corresponds to
\begin{equation}\label{partcase}
    W_1=\frac{1}{p}T^{1/2}\tilde{R}_1T^l(zI-\tilde{K})^{-1}T^{1/2},\hspace{0.5cm} W_2=...=W_{m+1}=\frac{1}{p}T^{1/2}\tilde{R}_1T^{1/2}
\end{equation}
The proof of (\ref{indbd}), i.e.,
\[\mathbb{E}[(\frac{1}{N}Z_1^TR_1Z_1)^m]=\mathbb{E}[(\frac{1}{N}tr(R_1))^m]+(Im(z))^{-m} \cdot [O((16Cm)^{m} \cdot (\frac{d}{n^{1/2}})^{1/3})+O(m^2 \cdot\frac{d^2}{n})],\]
shows that the first-order contributors are the tuples with all shared degrees vanishing. The argument relies on the growth of the moments of \(x_1,\) as well as 
\[p||W_1|| \leq \frac{1}{Im(z)}.\] 
In light of
\[p||W_1|| \leq ||T||^{l+1} \cdot ||\tilde{R}_1|| \cdot ||(zI-\tilde{K})^{-1}|| \leq ||T||^{l+1} \cdot \frac{1}{Im(z)} \cdot \frac{Im(z)+||T||}{(Im(z))^2},\]
\[p||W_2|| \leq ||T|| \cdot ||\tilde{R}_1|| \leq \frac{||T||}{Im(z)},\]
an almost verbatim rationale can be employed for (\ref{partcase}) since \(\gamma_n \cdot p||W_2|| \leq \frac{1}{4}\) for \(n \geq n_0\) (recall that \(l \in \{0,1\}\) are the sole values of interest). 
This entails that  
\[\mathbb{E}[Y(\frac{1}{p}\tilde{Z}_1^T\tilde{R}_1\tilde{Z}_1)^m]=\mathbb{E}[Y_G(\frac{1}{p}y^T\tilde{R}_1y_1)^m]+o(1)\]
because the first-order contributions are the same for both sides due to \((y_0)_1\) being subgaussian, symmetric with \(\mathbb{E}[((y_0)_1)^{2k}]=(2k-1)!! \leq (2k)^k\) for all \(k \in \mathbb{N}.\) 
Furthermore, the errors remain negligible when summed over \(m,\) providing the desired claim on \(\Sigma_{1e},\) i.e., (\ref{rel1}). Lastly, since the proofs of (\ref{ext1}) and (\ref{bulk2e}) remain valid when \(\tilde{Z}_1\) is replaced by \(y,\) claim (\ref{unmomeq}) is completely justified, and this concludes the proof of Proposition~\ref{propstep2e}.

\section{Auxiliary Results}\label{sectauxlemmas}

This section contains a few lemmas used throughout the paper.

\begin{lemma}\label{ledoitpechelemma}
Suppose \(Z_1,Z_2,\hspace{0.05cm}...\hspace{0.05cm},Z_p \in \mathbb{R}^N,\) let \(S=\frac{1}{p}\sum_{1 \leq j \leq p}{Z_jZ_j^T},\) and assume \(z \in \mathbb{C}\) with \(zI-S, \newline (zI-S+\frac{1}{p}Z_jZ_j^T)_{1 \leq j \leq p}\) invertible. Then 
    \[-1+z \cdot \frac{1}{N}tr((zI-S)^{-1})=-\frac{p}{N}+\frac{1}{N}\sum_{1 \leq j \leq p}{\frac{1}{1-\frac{1}{p}Z_j^TR_jZ_j}},\]
where \(R_j=(zI-S+\frac{1}{p}Z_jZ_j^T)^{-1},\) \(1 \leq j \leq p.\) 
\end{lemma}

\begin{proof}
    Begin with
\begin{equation}\label{740}
    -1+z\cdot \frac{1}{N}tr((zI-S)^{-1})=\frac{1}{pN}\sum_{1 \leq j \leq p}{Z_j^T(zI-S)^{-1}Z_j},
\end{equation}
obtained from multiplying
\[-(zI-S)+zI=\frac{1}{p}\sum_{1 \leq j \leq p}{Z_jZ_j^T}\]
on the right by \((zI-S)^{-1},\) taking the trace of each side (using \(tr(AB)=tr(BA)\) for \(A \in \mathbb{C}^{n \times m},B \in \mathbb{C}^{m \times n}\)), and dividing by \(N.\) The resolvent identity \(A^{-1}-B^{-1}=A^{-1}(B-A)B^{-1}\) subsequently gives
\[R_j-(zI-S)^{-1}=-R_j \frac{1}{p}Z_jZ_j^T (zI-S)^{-1},\]
after which multiplication with \(Z_j^T\) on the left and \(Z_j\) on the right yields
\[Z_j^TR_jZ_j-Z_j^T(zI-S)^{-1}Z_j=-\frac{1}{p}Z_j^TR_j Z_j \cdot Z_j^T (zI-S)^{-1}Z_j,\]
whereby
\begin{equation}\label{eq123}
    Z_j^T(zI-S)^{-1}Z_j=\frac{Z_j^TR_jZ_j}{1-\frac{1}{p}Z_j^TR_jZ_j}=-p+\frac{p}{1-\frac{1}{p}Z_j^TR_jZ_j}
\end{equation}
insomuch as
\[Z_j^T(zI-S)^{-1}Z_j \cdot (1-\frac{1}{p}Z_j^TR_jZ_j)=Z_j^TR_jZ_j,\]
entails \(1-\frac{1}{p}Z_j^TR_jZ_j \neq 0\) (else, \(Z_j^TR_jZ_j=0,\) but \(Z_j^TR_jZ_j=p,\) absurd). Lastly, the desired claim ensues from plugging (\ref{eq123}) in (\ref{740}), 
\[-1+z\cdot \frac{1}{N}tr((zI-S)^{-1})=-\frac{p}{N}
    +\frac{1}{N}\sum_{1 \leq j \leq p}{\frac{1}{1-\frac{1}{p}Z_j^TR_jZ_j}}.\]
\end{proof}

\begin{lemma}\label{linalglemma}
    Let \(u,v \in \mathbb{C}^n, A \in \mathbb{C}^{n \times n},\) and suppose \(A,A+uv^T\) are invertible. Then 
    \[(A+uv^T)^{-1}=A^{-1}-\frac{A^{-1}uv^TA^{-1}}{1+v^TA^{-1}u}.\]
\end{lemma}

\begin{proof}
    Multiply on the left with \(A\) and on the right with \(A+uv^T\) to transform the claimed identity into
    \[uv^T=\frac{uv^TA^{-1}(A+uv^T)}{1+v^TA^{-1}u}.\]
    This is a consequence of
\begin{equation}\label{helpp}
    v^TA^{-1}u \cdot uv^T=uv^TA^{-1}uv^T
\end{equation}
because this will then render 
\[(1+v^TA^{-1}u) \cdot uv^T=uv^TA^{-1}(A+uv^T),\] 
and \(1+v^TA^{-1}u \ne 0:\) otherwise, \(uv^TA^{-1}(A+uv^T)=0,\) from which\footnote{By a slight abuse of notation, zero vectors and matrices are denoted by \(0.\)} \(uv^T=0,\) 
giving rise to a contradiction,
\[-1=v^TA^{-1}u=tr(A^{-1}uv^T)=0.\] 
\par
Identity (\ref{helpp}) comes from 
\(f:\mathbb{C}^{n \times n} \to \mathbb{C}^{n \times n}\) defined by
\[f(M)=v^TMu \cdot uv^T-uv^TMuv^T\] 
vanishing everywhere (i.e., \(f=0\)) as it is linear and \(f(E_{\alpha \beta})=0\) for \(E_{\alpha \beta}=(\chi_{i=\alpha,j=\beta})_{1 \leq i,j \leq n},1 \leq \alpha,\beta \leq n.\) The latter holds insomuch as for \(1 \leq i,j \leq n,\)
\[(f(E_{\alpha \beta}))_{ij}=v_\alpha u_\beta \cdot u_iv_j-(uv^T)_{i\alpha} \cdot (uv^T)_{\beta j}=v_\alpha u_\beta \cdot u_iv_j-u_iv_{\alpha} \cdot u_\beta v_{j}=0.\] 
\end{proof}

\begin{lemma}\label{lenlemma}
    Suppose \(Z_0\) is given by (\ref{zmoddel}), and let \(N=\binom{n}{d}.\) 
    If \(d \leq \frac{n}{2}\) and \(\mathbb{E}[x_1^4] \leq \frac{n-2d+2}{(d-1)^2},\) then 
    \[Var(||Z_0||^2) \leq 2N^2 \cdot \mathbb{E}[x_1^4] \cdot \frac{d^2}{n} \cdot e^{-\frac{(d-1)^2}{n-1}}.\]
\end{lemma}
\(\newline\)
\textbf{Remark:} Since \(\mathbb{E}[x_1^4] \geq (\mathbb{E}[x_1^2])^2=1,\) the fourth moment condition entails \(d^2 \leq n+1,\) which is stronger than \(d \leq \frac{n}{2}\) so long as \(n \geq 5\) from \(\frac{n^2}{4}-(n+1)=\frac{(n-2)^2}{4}-2.\)

\begin{proof}
    Linearity of expectation and \(x_1,x_2,\hspace{0.05cm}...\hspace{0.05cm},x_n\) being i.i.d. yield
    \[\mathbb{E}[||Z_0||^2]=N \cdot \mathbb{E}[x_1^2x_2^2...x_d^2]=N,\]
    whereby
    \[Var(||Z_0||^2)=\mathbb{E}[(||Z_0||^2-N)^2]=\mathbb{E}[(\sum_{i_1<i_2<...<i_d}{(x_{i_1}^2x_{i_2}^2...x_{i_d}^2-1)})^2].\]
    The pairs making nonzero contributions have their two underlying tuples share \(k\) entries for some \(1 \leq k \leq d,\) whereby 
    \begin{equation}\label{varcomp}
        Var(||Z_0||^2)=\binom{n}{d}\sum_{1 \leq k \leq d}{((\mathbb{E}[x_1^4])^k-1) \cdot \binom{d}{k} \binom{n-d}{d-k}} \leq \binom{n}{d}\sum_{1 \leq k \leq d}{(\mathbb{E}[x_1^4])^k \cdot \binom{d}{k} \binom{n-d}{d-k}}
    \end{equation}
    (fix the first tuple, choose a subset of size \(k\) from the set formed by its entries, and select a subset of size \(d-k\) of a set of size \(n-d\) for the remaining positions of the second tuple). 
    Since \(d-1<n-d,\) the ratio between the terms corresponding to \(k+1\) and \(k\) is
    \[\mathbb{E}[x_1^4] \cdot \frac{\binom{d}{k+1} \binom{n-d}{d-k-1}}{\binom{d}{k} \binom{n-d}{d-k}}=\mathbb{E}[x_1^4] \cdot \frac{d-k}{k+1} \cdot \frac{d-k}{n-2d+k+1},\]
    for \(1 \leq k \leq d-1.\) This product decreases in \(k\) (the numerator decreases and the denominator increases in \(k\)), 
    entailing the maximum is attained at \(k=1,\) at which it is 
    \[\mathbb{E}[x^4] \cdot \frac{(d-1)^2}{2(n-2d+2)} \leq \frac{1}{2}.\] 
    The desired bound 
    ensues from \(\sum_{k \geq 1}{2^{-(k-1)}}=2,\) and
    \[\frac{\binom{n}{d} \cdot d\binom{n-d}{d-1}}{N^2}=\frac{d\binom{n-d}{d-1}}{\binom{n}{d}}=\frac{d^2(n-d)!(n-d)!}{n!(n-2d+1)!}=\frac{d^2(n-d)...(n-2d+2)}{n(n-1)...(n-d+1)}=\]
    \[=\frac{d^2}{n} \cdot \prod_{1 \leq k \leq d-1}{\frac{n-2d+1+k}{n-d+k}}=\frac{d^2}{n} \cdot \prod_{1 \leq k \leq d-1}{(1-\frac{d-1}{n-d+k})} \leq \frac{d^2}{n} \cdot e^{-\sum_{1 \leq k \leq d-1}{\frac{d-1}{n-d+k}}} \leq \frac{d^2}{n} \cdot e^{-\frac{(d-1)^2}{n-1}}\]
    via \(1-x \leq e^{-x}\) for \(x \in \mathbb{R}.\)
    \vspace{0.3cm}
    \par
    Incidentally, the last inequality is tight (in a ratio sense) when \(\mathbb{E}[x^4]-1\) is bounded away from \(0\) and \(d \leq \frac{n}{5}\) because the arguments above and \(\mathbb{E}[x_1^4] \geq (\mathbb{E}[x_1^2])^2=1\) readily imply that
    \[ Var(||Z_0||^2) \geq \binom{n}{d} \cdot d\binom{n-d}{d-1} \cdot (\mathbb{E}[x^4]-1)=N^2 \cdot (\mathbb{E}[x^4]-1) \cdot \frac{d^2}{n} \cdot \prod_{1 \leq k \leq d-1}{(1-\frac{d-1}{n-d+k})} \geq\]
    \begin{equation}\label{lbvar}
        \geq N^2 \cdot (\mathbb{E}[x^4]-1) \cdot \frac{d^2}{n} \cdot e^{-2\sum_{1 \leq k \leq d-1}{\frac{d-1}{n-d+k}}} \geq N^2 \cdot (\mathbb{E}[x^4]-1) \cdot \frac{d^2}{n} \cdot e^{-\frac{2(d-1)^2}{n-d+1}}
    \end{equation}
    by using that \(1-x \geq e^{-2x}\) for \(x \leq \frac{\log{2}}{2}\) from \((1-x-e^{-2x})'=2e^{-2x}-1 \geq 0,\) as well as 
    \[\frac{d-1}{n-d+k} \leq \frac{d-1}{n-d+1} \leq \frac{n/5-1}{n-n/5+1}<\frac{1}{4}<\frac{\log{2}}{2} \hspace{0.5cm} (1 \leq k \leq d-1).\] 
\end{proof}

The following result entails that when \(\frac{d}{n^{1/2}}\) is away from zero, the behavior of \(Var(||Z_0||^2)\) can change significantly from the one described in Lemma~\ref{lenlemma}.

\begin{lemma}\label{lenlemma2}
   Suppose \(Z_0\) is given by (\ref{zmoddel}), and let \(N=\binom{n}{d},B=\mathbb{E}[x_1^4].\) If \((2n)^{1/2} \leq d \leq \frac{n}{8},\) then 
    \[Var(||Z_0||^2) \geq N^2 \cdot \frac{1-B^{-\frac{d^2}{n}}}{8} \cdot (\frac{B}{8e^4})^{\frac{d^2}{n}}.\]
\end{lemma}

\begin{proof}
    Let \(k=\lceil \frac{d^2}{n} \rceil.\) Then 
    \[2 \leq k \leq \frac{d^2}{n}+1 \leq \frac{2d^2}{n} \leq \frac{d}{2},\] 
    and (\ref{varcomp}) together with \(B=\mathbb{E}[x^4] \geq (\mathbb{E}[x^2])^2=1\) gives
    \[Var(||Z_0||^2) \geq N^2 \cdot (B^k-1) \cdot \frac{\binom{d}{k} \binom{n-d}{d-k}}{\binom{n}{d}}.\]
    By virtue of \(1-B^{-k} \geq 1-B^{-\frac{d^2}{n}},\) it suffices to show
    \[\frac{\binom{d}{k} \binom{n-d}{d-k}}{\binom{n}{d}} \geq 8^{-1} \cdot (8e^4)^{-\frac{d^2}{n}}:\]
    this ensues from
    \[\frac{\binom{d}{k} \binom{n-d}{d-k}}{\binom{n}{d}} \geq \frac{\frac{(d-k)^k}{k!} \cdot \frac{(n-2d)^{d-k}}{(d-k)!}}{\frac{n^d}{d!}}=(1-\frac{2d}{n})^{d-k} \cdot \binom{d}{k} \cdot (\frac{d-k}{n})^{k} \geq e^{-\frac{4d(d-k)}{n}} \cdot (\frac{(d-k)^2}{kn})^k \geq\]
    \[\geq e^{-\frac{4d^2}{n}} \cdot (\frac{d^2/4}{\frac{2d^2}{n} \cdot n})^k= e^{-\frac{4d^2}{n}} \cdot 8^{-k} \geq 8^{-1} \cdot (8e^4)^{-\frac{d^2}{n}}\]
    via \(1-x \geq e^{-2x}\) for \(x \leq \frac{\log{2}}{2},\) and \(\frac{2d}{n} \leq \frac{1}{4}< \frac{\log{2}}{2}.\)
\end{proof}

\begin{lemma}\label{lemmamom}
    Suppose \(\mathbb{E}[x^2]=1,\mathbb{E}[x^{2q}] \leq (Cq)^q\) for all \(q \in \mathbb{N},\) and let 
    \[c(m,n,(d_j)_{1 \leq j \leq m})=\frac{1}{\binom{n}{d_1} \cdot \binom{n}{d_2} \cdot ... \cdot \binom{n}{d_m}}\sum_{1 \leq i_j \leq \binom{n}{d_j}}{\mathbb{E}[Z^2_{d_1,p_1}Z^2_{d_2,p_2}...Z^2_{d_m,p_m}]}\]
    for \(m \in \mathbb{N}, d_1,d_2, \hspace{0.05cm}...\hspace{0.05cm},d_m \in \{1,2,\hspace{0.05cm}...\hspace{0.05cm},n\},\) \(x_1,x_2,\hspace{0.05cm}...\hspace{0.05cm},x_n,\) i.i.d., \(x_1\overset{d}{=} x,\) and \((Z_{l,k})_{1 \leq k \leq \binom{n}{l}}\) a fixed labeling of 
    \((x_{r_1}x_{r_2}...x_{r_l})_{1 \leq r_1<r_2<...<r_l \leq n}.\)
    Then
    \begin{equation}\label{wantedbd}
        c(m,n,(d_j)_{1 \leq j \leq m}) \leq 1+e^{\frac{2C^2e^2(\sum_{1 \leq j \leq m}{d_j})^2}{n}} \cdot \frac{2C^2e^2(\sum_{1 \leq j \leq m}{d_j})^2}{n}
    \end{equation}
    when \(\sum_{1 \leq j \leq m}{d_j} \leq \frac{n^{1/2}}{3Ce}.\)  
\end{lemma}

\begin{proof}
    If there is no overlap among the positions underlying \(p_1,p_2,\hspace{0.05cm}...\hspace{0.05cm},p_m,\) then the expectations contribute at most \(1\) from independence and \(\mathbb{E}[Z^2_{d_j,p_j}]=1\) for \(1 \leq j \leq m.\) Else, there is an index \(r \in \{1,2,\hspace{0.05cm}...\hspace{0.05cm},n\}\) that appears exactly in the sets underpinning the positions \(p_{i_1},p_{i_2},\hspace{0.05cm}...\hspace{0.05cm},p_{i_k}\) for \(1 \leq i_1<i_2<...<i_k \leq m\) and some \(2 \leq k \leq m.\) The contribution of \(x_r\) to the expectation is, by independence, \(\mathbb{E}[x^{2k}],\) and
    \[\binom{n}{d}=\binom{n-1}{d-1} \cdot \frac{n}{d} \hspace{0.5cm} (1 \leq d \leq n), \hspace{0.8cm} \binom{n}{d}=\binom{n-1}{d} \cdot \frac{n}{n-d} \hspace{0.5cm} (1 \leq d \leq n-1)\]
    imply
    \begin{equation}\label{ineqq1}
        c(m,n,(d_j)_{1 \leq j \leq m}) \leq 1+\sum_{2 \leq k \leq m}{n \cdot \mathbb{E}[x^{2k}]\sum_{1 \leq i_1<...<i_k \leq m}{\frac{1}{\frac{n}{d_{i_1}} \cdot \frac{n}{d_{i_2}} \cdot ... \cdot \frac{n}{d_{i_k}}}} \cdot c(m,n-1,(d_j-\chi_{j \in \{i_1,i_2, \hspace{0.05cm}...\hspace{0.05cm},i_k\}})_{1 \leq j \leq m})}.
    \end{equation}
    The multinomial theorem gives 
    \[\sum_{1 \leq i_1<...<i_k \leq m}{\frac{1}{\frac{n}{d_{i_1}} \cdot \frac{n}{d_{i_2}} \cdot ... \cdot \frac{n}{d_{i_k}}}} \leq (\sum_{1 \leq j \leq m}{\frac{1}{\frac{n}{d_j}}})^k \cdot \frac{1}{k!}=(\frac{\sum_{1 \leq j \leq m}{d_j}}{n})^k \cdot \frac{1}{k!},\] 
    from which (\ref{ineqq1}) can be changed to
    \[c(m,n,(d_j)_{1 \leq j \leq m}) \leq 1+\sum_{2 \leq k \leq m}{n \cdot \mathbb{E}[x^{2k}] \cdot (\frac{\sum_{1 \leq j \leq m}{d_j}}{n})^k \cdot \frac{1}{k!} \cdot \max_{1 \leq i_1<...<i_k \leq m}}{c(m,n-1,(d_j-\chi_{j \in \{i_1,i_2, \hspace{0.05cm}...\hspace{0.05cm},i_k\}})_{1 \leq j \leq m})}.\]
    \par
    Proceed by induction on \(s=\sum_{1 \leq j \leq m}{d_j} \leq \frac{n^{1/2}}{3Ce}.\) Note that \(n \geq 2\) from \(s \geq 1,3Ce \geq 3e>1\) as
    \begin{equation}\label{ceq}
        C \geq \mathbb{E}[x^2]=1.
    \end{equation}
    The base case \(s=1\) has \(m=1,\) and the result is clear (the sum is at most \(1\)). Suppose next the bound holds for \(s-1 \geq 1.\) By virtue of \(k! \geq (\frac{k}{e})^k,\) it suffices to show that\footnote{The bound in (\ref{wantedbd}) does  not depend directly on \(m,\) and thus the cases in which \(d_{i_j}=1\) for some \(1 \leq j \leq k\) cause no issue (such positions can be dropped from the tuple corresponding to \(n-1\)).}
    \begin{equation}\label{eqsum}
        1+\sum_{2 \leq k \leq m}{n \cdot (Ck)^{k} \cdot (\frac{es}{nk})^k \cdot (1+e^{\frac{2C^2e^2(s-k)^2}{n-1}} \cdot \frac{2C^2e^2(s-k)^2}{n-1})}:=1+I+II
    \end{equation}
    is upper bounded by the right-hand side term of (\ref{wantedbd}) (using as well that \(x \to xe^{x}\) increases on \([0,\infty)\)), where \(I,II\) are obtained by opening the parentheses in the last factors of the summands since 
     \[s-k \leq \frac{n^{1/2}}{3Ce}-2 \leq \frac{(n-1)^{1/2}}{3Ce}\] 
     from 
     \[\frac{n^{1/2}-(n-1)^{1/2}}{3Ce}=\frac{1}{3Ce \cdot (n^{1/2}+(n-1)^{1/2})} \leq \frac{1}{3Ce} \leq \frac{1}{3e}<1\]
     via (\ref{ceq}). 
     \par
     Rewrite the factors in (\ref{eqsum}) as
    \[n \cdot (Ck)^{k} \cdot (\frac{es}{nk})^k=n \cdot (\frac{Ces}{n})^{k}=\frac{C^2e^2s^2}{n} \cdot (\frac{Ces}{n})^{k-2},\]
    from which it is enough to show that
    \[I=\frac{C^2e^2s^2}{n} \sum_{2 \leq k \leq m}{(\frac{Ces}{n})^{k-2}} \leq e^{\frac{2C^2e^2s^2}{n}} \cdot \frac{C^2e^2s^2}{n},\]
    \[II=\frac{C^2e^2s^2}{n} \sum_{2 \leq k \leq m}{(\frac{Ces}{n})^{k-2} \cdot e^{\frac{2C^2e^2(s-k)^2}{n-1}} \cdot \frac{2C^2e^2(s-k)^2}{n-1}} \leq e^{\frac{2C^2e^2s^2}{n}} \cdot \frac{C^2e^2s^2}{n}\]
    when \(m \leq s \leq \frac{n^{1/2}}{3Ce}\) to complete the induction step. 
    \par
    Begin with \(I:\) the desired claim is equivalent to
    \[\sum_{0 \leq k \leq m-2}{(\frac{Ces}{n})^{k}} \leq e^{\frac{2C^2e^2s^2}{n}},\]
    and ensues from
    \begin{equation}\label{helper}
        \sum_{0 \leq k \leq M}{x^k} \leq e^{Mx} \hspace{1cm} (x \geq 0,M \geq 0),
    \end{equation}
    a consequence of
    \[e^{Mx}=\sum_{k \geq 0}{\frac{(Mx)^k}{k!}} \geq \sum_{k \geq 0}{\frac{(Mx)^k}{k^k}} \geq \sum_{0 \leq k \leq M}{x^k},\]
    under the convention \(0^0:=1.\) Namely, (\ref{helper}) gives
    \[\sum_{0 \leq k \leq m-2}{(\frac{Ces}{n})^{k}} \leq e^{\max{(m-2,0)} \cdot \frac{Ces}{n}}<e^{s \cdot \frac{Ces}{n}}< e^{\frac{2C^2e^2s^2}{n}}\]
    by using \(s \geq m\) and (\ref{ceq}). 
    \par
    Consider now \(II,\) for which it suffices to justify
    \[\frac{C^2e^2s^2}{n}\sum_{2 \leq k \leq m}{(\frac{Ces}{n})^{k-2} \cdot e^{\frac{2C^2e^2(s-k)^2}{n-1}}} \leq e^{\frac{2C^2e^2s^2}{n}} \cdot \frac{n-1}{2n}\]
    from \((s-k)^2 \leq s^2\) for \(k \leq m \leq s.\) Since \((s-k)^2 \leq s(s-k),\) and the corresponding terms, after replacing the squares in the exponentials by these linear bounds in \(k,\) satisfy
    \[\frac{t(k+1)}{t(k)}=\frac{Ces}{n} \cdot e^{-\frac{2C^2e^2s}{n-1}}< \frac{n^{1/2}}{3n} \leq \frac{1}{3}.\]
    This renders that the sum is at most \(\frac{1}{1-\frac{1}{3}} \cdot t(2) \leq 2t(2),\) amounting to an overall bound of at most
    \[\frac{C^2e^2s^2}{n} \cdot 2 \cdot e^{\frac{2C^2e^2s(s-2)}{n-1}} \leq e^{\frac{2C^2e^2s^2}{n}} \cdot \frac{n-1}{2n}:\]
    to derive this last inequality, use
    \[e^{\frac{2C^2e^2s(s-2)}{n-1}}< e^{\frac{2C^2e^2s^2}{n}}\]
    from 
    \[\frac{s-2}{n-1}-\frac{s}{n}=\frac{n(s-2)-s(n-1)}{n(n-1)}=\frac{s-2n}{n(n-1)} \leq \frac{-n}{n(n-1)}<0\] 
    as \(s \leq \frac{n^{1/2}}{3Ce} \leq \frac{n^{1/2}}{3e}<n,\) while 
    \[\frac{C^2e^2s^2}{n} \cdot 2 \cdot \frac{2n}{n-1} \leq C^2e^2 \cdot \frac{1}{9C^2e^2}\cdot 2 \cdot 4=\frac{8}{9}<1,\]
    completing the induction step (recall that \(n \geq 2\)).
    
\end{proof}

\bibliography{}

\end{document}